\documentclass[11pt,reqno]{amsart}
\usepackage{amssymb,latexsym}
\usepackage{graphicx}
\usepackage{url}		%does nice formatting of URLs

\calclayout

\makeatletter
\newcommand{\monthyear}[1]{%
  \def\@monthyear{\uppercase{#1}}}
\newcommand{\volnumber}[1]{%
  \def\@volnumber{\uppercase{#1}}}
\AtBeginDocument{%
\def\ps@plain{\ps@empty
  \def\@oddfoot{\@monthyear \hfil \thepage}%
  \def\@evenfoot{\thepage \hfil \@volnumber}}
\def\ps@firstpage{\ps@plain}
\def\ps@headings{\ps@empty
  \def\@evenhead{%
    \setTrue{runhead}%
    \def\thanks{\protect\thanks@warning}%
    \uppercase{The Fibonacci Quarterly}\hfil}%
  \def\@oddhead{%
    \setTrue{runhead}%
    \def\thanks{\protect\thanks@warning}%
    \hfill\uppercase{Primes and composites in the Det Hosoya triangle}}%
  \let\@mkboth\markboth
  \def\@evenfoot{%
    \thepage \hfil \@volnumber}%
  \def\@oddfoot{%
    \@monthyear \hfil \thepage}%
  }%
\footskip=25pt
\pagestyle{headings}%
}
\makeatother

\theoremstyle{plain}
\numberwithin{equation}{section}
\newtheorem{thm}{Theorem}[section]
\newtheorem{theorem}[thm]{Theorem}
\newtheorem{lemma}[thm]{Lemma}

\newtheorem{proposition}[thm]{Proposition}
\newtheorem{corollary}[thm]{Corollary}

\usepackage{cite}
\usepackage[usenames]{color}
\definecolor{webgreen}{rgb}{0,.5,0}
\definecolor{webbrown}{rgb}{.6,0,0}

\newcommand{\legendre}[2]{\genfrac{(}{)}{}{}{#1}{#2}}

\begin{document}
%% replace the values in the next three lines by the correct information
\monthyear{Month Year}
\volnumber{Volume, Number}
\setcounter{page}{1}

\title{Primes and composites in the determinant Hosoya triangle}
\author{Hsin-Yun Ching}
\address{Department of Mathematical Sciences\\
                The Citadel\\
                Charleston,  SC\\
                USA}
\email{george425tw@gmail.com}
\thanks{The second and fourth authors were partially supported by The Citadel Foundation. Several of the main results in this paper were found by 
the first and second authors while the first was working on his undergraduate research project under the guidance of the second and fourth authors (who followed the 
guidelines given  in \cite{FlorezMukherjeeED}).}
\author{Rigoberto Fl\'orez}
\address{Department of Mathematical Sciences\\
                The Citadel\\
                Charleston,  SC\\
                USA}
\email{rigo.florez@citadel.edu}
\author{F. Luca}
	\address{School of Maths, Wits University, South Africa, Research Group in Algebraic Structures and Applications, King Abdulaziz University, Saudi Arabia and Centro de Ciencias Matem\'aticas UNAM, Morelia, Mexico}
	\email{florian.luca@wits.ac.za}
\author{Antara Mukherjee}
\address{Department of Mathematical Sciences\\
                The Citadel\\
                Charleston,  SC\\
                USA}
\email{antara.mukherjee@citadel.edu}
\author{J.~C.~Saunders}
\address{Department of Mathematics and Statistics\\ 
               University of Calgary\\
               Calgary, AB T2N 1N4\\
                Canada.}
\email{john.saunders1@ucalgary.ca}

\begin{abstract}
In this paper, we look at numbers of the form $H_{r,k}:=F_{k-1}F_{r-k+2}+F_{k}F_{r-k}$. These numbers are the entries of a triangular array called the 
\emph{determinant Hosoya triangle} which we denote by ${\mathcal H}$. We discuss the divisibility properties of the above numbers and their primality. We give a 
small  sieve of primes to illustrate the density of prime numbers in ${\mathcal H}$. Since the Fibonacci and Lucas numbers appear as entries in ${\mathcal H}$, 
our research is an extension of the classical questions concerning whether there are infinitely many Fibonacci or Lucas primes. We prove that ${\mathcal H}$ has arbitrarily large neighbourhoods of composite entries.  
Finally we present an abundance of data indicating a very high density of primes in ${\mathcal H}$.

\end{abstract}

\maketitle
\section {Introduction}
 The \emph{determinant Hosoya triangle} is a triangular array denoted by ${\mathcal H}$  that has the same recurrence relation of the Hosoya triangle \cite{hosoya} but with different initial   
 conditions. The entries of this triangle can be expressed as determinants of Fibonacci numbers. Thus, the  entries of the determinant Hosoya triangle are given by 
 \begin{equation}
 \label{eq:maineq}
 H_{r,k}=\begin{vmatrix}
F_{k+1} & F_{k}\\
F_{r-k+1} & F_{r-k+2}\\
\end{vmatrix},
\end{equation}
where $r\ge 1$ represents the line and $k\ge 1$ represents the position in the line $r$ from left to right, see Table \ref{tabla2}. 
For example, the  entry $H_{7,5}$ of $\mathcal{H}$ is given by

\[H_{7,5}=\begin{vmatrix}
F_{6} & F_{5}\\
F_{3} & F_{4}\\
\end{vmatrix}=\begin{vmatrix}
8 & 5\\
2 & 3\\
\end{vmatrix}=24-10=14.\]

\begin{table} [!ht] \small
	\centering
	\addtolength{\tabcolsep}{-3pt} \scalebox{.85}{
		\begin{tabular}{ccccccccccccccccccccccccccc}
   &&&&&&&&&&&&&     	                                $0$                                                &&&&&&&&&&&&&\\
   &&&&&&&&&&&&     	                           $1$  &&   $1$                                            &&&&&&&&&&&&\\
   &&&&&&&&&&&     		           $1$  &&   $3$  &&   $1$                                        &&&&&&&&&&&\\
   &&&&&&&&&&     	                     $2$  &&   $4$  &&   $4$  &&   $2$                                    &&&&&&&&&&\\
   &&&&&&&&&     		        $3$  &&   $7$  &&   $5$  &&   $7$  &&   $3$                                &&&&&&&&&\\
   &&&&&&&&     	           $5$  &&   $11$ &&   $9$  &&   $9$  &&   $11$ &&   $5$                            &&&&&&&&\\
   &&&&&&&     	          $8$  &&   $18$ &&   $14$ &&   $16$ &&   $14$ &&   $18$ &&   $8$                        &&&&&&&\\
   &&&&&&     	     $13$ &&   $29$ &&   $23$ &&   $25$ &&   $25$ &&   $23$ &&   $29$ &&   $13$                   &&&&&&\\
   &&&&&        $21$ &&   $47$ &&   $37$ &&   $41$ &&   $39$ &&   $41$ &&   $37$ &&   $47$ &&   $21$              &&&&&\\
   &&&&    $34$ &&   $76$ &&   $60$ &&   $66$ &&   $64$ &&   $64$ &&   $66$ &&   $60$ &&   $76$ &&   $34$          &&&&\\
   &&&	$55$ &&   $123$&&   $97$ &&   $107$&&   $103$&&   $105$&&   $103$&&   $107$&&   $97$ &&   $123$&&  $55$ &&&     \\
  &&$89$ && $199$ && $157$ && $173$ && $167$ && $169$ && $169$ && $167$ && $173$ && $157$ && $199$ && $89$  &&  \\
  &  $144$ && $322$ && $254$ && $280$ && $270$ && $274$ && $272$ && $274$ && $270$ && $280$ && $254$ && $322$ && $144$&\\
  $233$ && $521$ && $411$ && $453$ && $437$ && $443$ && $441$ && $441$ && $443$ && $437$ && $453$ && $411$ && $521$ && $233$\\
\end{tabular}}
	\caption{The determinant Hosoya triangle  $\mathcal{H}$.} \label{tabla2}
\end{table}

Expanding the determinant in \eqref{eq:maineq}, we get 
\begin{eqnarray*}
H_{r,k} & = & F_{k+1}F_{r-k+2}-F_kF_{r-k+1}=(F_k+F_{k-1})F_{r-k+2}-F_kF_{r-k+1}\\
& = & F_{k-1}F_{r-k+2}+F_k(F_{r-k+2}-F_{r-k+1})=F_{k-1}F_{r-k+2}+F_kF_{r-k}.
\end{eqnarray*}
Further, the entries of ${\mathcal H}$ have the generating function
$$
\frac{(x+y+xy)}{(1-x-x^2)(1-y-y^2)}=\sum_{r\ge 1,k\ge 1} H_{r,k} x^ry^k
$$
originally discovered by Sloane \cite[A108038]{sloane}.

We now give a recursive definition of ${\mathcal H}$. Namely, ${\mathcal H}$ is the doubly indexed sequence $\{H_{r,k}\}_{\substack{r\ge 1\\ k\ge 1}}$ given by

\begin{equation}\label{Hosoya:Seq}
 H_{r,k}= H_{r-1,k}+H_{r-2,k}  \; \text{ and } \;
 H_{r,k}= H_{r-1,k-1}+H_{r-2,k-2}, \quad r\ge 3\quad {\text{\rm and}}\quad 1\le k\le r
 \end{equation}
with initial conditions $H_{1,1}=0, H_{2,1}= H_{2,2}=1,$ and $ H_{3,2}=3$. 
Note that the left hand-side of \eqref{Hosoya:Seq} gives of rise to slash diagonals and the right hand-side gives rise to backslash diagonals of the triangle.

It is easy to see that the entries (points) of every diagonal (slash or backslash) in ${\mathcal H}$ correspond to a 
generalized Fibonacci number, in the sense that the sequence of such entries obeys the Fibonacci recurrence in one of the two parameters when the other one is fixed. For instance, the sequence 
$$
3,~ 11,~ 14,~ 25,~ 39,~ 64,~ \dots .
$$
that corresponds to the 
fifth diagonal  in Table \ref{tabla2} obeys the Fibonacci recurrence  
$G_{n}^{(5)}=G_{n-1}^{(5)}+ G_{n-2}^{(5)}$, with $G_{1}^{(5)}=3$ and $G_{2}^{(5)}=11$. 
In general, the entries of the $m$th diagonal of this triangle are given by
\begin{equation}\label{HoyaAsGenFibo}
G_{n}^{(m)}=G_{n-1}^{(m)}+ G_{n-2}^{(m)}, \quad  \text{ where }  \quad  G_{1}^{(m)}= F_{m-1} \quad  \text{ and }  \quad   G_{2}^{(m)}=L_{m}.
\end{equation}
For example, for a fixed $m$ the slash diagonal in position $m$ in the triangle is given by  
$$
G_{n}^{(m)}=H_{n,m}= L_{n-m+1}F_{m-1}+F_{n-m}F_{m-2}.
$$
For a fixed $m$ the backslash diagonal in position $m$ in the triangle is given by  
$$
G_{n}^{(m)}=H_{n+m-1,m}= L_m F_{n-1}+F_{m-1} F_{n-2}.
$$
Other equivalent identities are 
\begin{equation}\label{definitionHrk1}
H_{r,k}=F_{k-1}F_{r-k+2}+F_{k}F_{r-k},
\end{equation}
and 
\begin{equation}\label{definitionHrk2}
H_{r,k}= F_{k+1} F_{r-k+2}- F_{k}F_{r-k+1}.
\end{equation}

Blair et al. \cite{BlairRigoAntara} give several combinatorial connections with the determinant Hosoya  triangle as well 
as combinatorial interpretation of its entries. More properties of this triangle and other similar triangles can be found in 
\cite{BlairRigoAntaraHoneycombs, BlairRigoAntaraGP, BlairRigoAntara, Blair, Ching, Czabarka, florezHiguitaMukherjee, florezHiguitaJunesGCD, florezjunes, florezHiguitaMukherjee2, FlorezMukherjeeED, hosoya, koshy2}. Recently, Benjamin et al. \cite{Benjamin} gave several elegant combinatorial proofs of identities from the Hosoya Triangle. It will be very interesting to see similar results for the determinant Hosoya triangle. 
 
Number theorists have been interested in Fibonacci primes for a very long time, at least starting with Lucas' seminal paper \cite{Luc} that investigated properties of Lucas sequences, which are natural generalizations of the Fibonacci sequence 
 with the aim of developing primality tests. Some results on Fibonacci primes can be found  
in \cite{jarden, brillhart}. In this paper we ask whether there are infinitely many primes of the form $H_{r,k}$. When $k=1$ and $k=2$ we get the classical questions 
concerning the existence of infinitely many primes in the Fibonacci and Lucas sequences, respectively. For example, from Table \ref{tabla2} we observe that:
\begin{itemize}
\item in row $3$ there is one prime number, $H_{3,2}=3$; 
\item in row $4$  there are two prime numbers $H_{4, 1}= 2$ and $H_{4, 4}= 2$; 
\item in row $5$ there are five prime numbers,   $H_{5, 1}= 3$,  $H_{5, 2}= 7$,  $H_{5, 3}= 5$, $H_{5, 4}= 7$, and  $H_{5, 5}= 3$; 
\item in row $6$ there are four primes
 $H_{6, 1}= 5$, $H_{6, 2}= 11$,  $H_{6, 5}= 11$, and $H_{6, 6}= 5$.  
 \end{itemize}
The distinct primes from Table \ref{tabla2} in the order that they appear are
$$3, ~2,~ 7,~ 5,~ 11,~ 13,~ 29,~ 23,~ 47,~ 37,~ 41,~ 97,~ 107,~ 103,~ 89, ~199,~ 157,~ 173,~ 167, ~233.
$$ 
From the same table we see that  
$$
H_{r+3,3}=F_{r}+L_{r+1}.
$$ 
There are several examples of primes that are the sum of a Fibonacci number and a Lucas number  
with consecutive subscripts such as:   
$$
5,~ 23,~ 37,~ 97,~ 157,~ 1741,~ 11933,~ 50549
$$ 
(for a longer list of such known primes see \cite[A091157]{sloane}). Are there infinitely many primes of this form? The answer to this question is not known. 

Searching for primes of the form $H_{r+3,3}=F_{r}+L_{r+1}$ we ran the computer algebra package \emph{Mathematica} until $r=80,000$, 
 found that there are $47$ primes of this form. We did not find any new prime for $r$ between $64,000$ and $80,000$. 
Running \emph{Mathematica} until $r=80,000$ for primes of the form $F_{r}+L_{r-1}$ (see \cite[A153892]{sloane}) we have found that there are $43$ primes of this form.  
We did not find any new prime for $r$ between $64,000$ and $80,000$. 
So, the formal question states as: Are there infinitely many primes of the form $F_{r}+L_{r\pm1}$? 
If $F_r+L_{r+1}$ and $F_r+L_{r-1}$  are primes for a fixed $r$, then they are called \emph{neighboring-Lucas-Fibonacci primes}. The two known neighboring-Lucas-Fibonacci primes are $2$, $5$ and $19$, $37$, is there any other pair of neighboring-Lucas-Fibonacci primes? Note that 13 is the common primitive root for these 4 primes. 

In this paper, we use divisibility properties of entries of ${\mathcal H}$ in order to give a short sieve for their primality. We also prove that there are arbitrarily 
large neighborhoods of ${\mathcal H}$ where all entries are composite.  This can be seen as a two-dimensional analogue of the well--known fact that there 
are arbitrarily large intervals of composite numbers, such as $[m!+2, m!+m]$ with an arbitrary integer $m\ge 3$.  We also give the data of all prime entries of 
${\mathcal H}$ found with $r\le 6000$.

\section{Some divisibility properties of numbers of the form $H_{r,k}$ }\label{congurencesHrk}

In this section, we give some divisibility properties of entries of ${\mathcal H}$. For the calculations in this section it is convenient to extend the Fibonacci numbers to negative indices via the recurrence 
$$
F_{-n}=F_{-n+2}-F_{-n+1}\quad {\text{\rm for}}\quad n\ge 0.
$$
It is well-known that $F_{-n}=(-1)^{n-1} F_n$ for all $n\ge 0$. 

\subsection{Congruences}
First of all we prove that two consecutive (on the first variable) numbers of the form  $H_{r,k}$ are relative prime. 

\begin{proposition} \label{RelativePrimes}
For $1\le k\le n$  we have $\gcd(H_{r,k}, H_{r-1,k})=1$ and $\gcd(H_{r,k}, H_{r-1,k-1})=1$.
\end{proposition}

\begin{proof} These follow from the identity
$$(-1)^rH_{r,k+1}H_{r,k}+ (-1)^{r+1} H_{r+1,k+1}H_{r-1,k} = 1, 
$$
from Blair \cite{BlairRigoAntaraGP}. 
\end{proof}

More generally, using \eqref{HoyaAsGenFibo} and \cite[Theorem 5]{florezHiguitaJunesGCD} the previous result generalizes to the following proposition. 

\begin{proposition} \label{GCDProperty}
For $1\le k\le r$ and $m\ge 1,~w\ge 0$ we have  
$$
\gcd(H_{r,m}, H_{r+w,m})\mid F_{w}\quad {\text{\rm and}}\quad \gcd(H_{r+m,m}, H_{r+m+w,m+w})\mid F_{w}.
$$
\end{proposition}

\begin{proposition}[Proposition 2.1 \cite{Ching}]\label{LinesWithCoposite}
If $1\le k \le r$ are positive integers, then 
$$H_{r,k} \equiv  0 \bmod F_{\gcd(r+2,k+1)}\quad \text{ and } \quad  H_{r,k} \equiv  0 \bmod F_{\gcd(r+2,k)}.$$
\end{proposition}

The proof of the following proposition follows using \cite[Proposition 2.1]{Ching}. However, here we give a self-contained proof.  

\begin{proposition} \label{CongruencesMod2}
The point $H_{r,k} \equiv  0 \bmod 2$ if and only if $r \equiv 1 \bmod 3$ for $r\ge 1$.
\end{proposition}

\begin{proof} For fixed $k$, $\{H_{r,k}\}_{r\ge 0}$ obeys a Fibonacci recurrence. So, in particular it is periodic in $r$ modulo $2$ with period $3$. Calculating $H_{r,k}$ for $r=0,1,2$ modulo $2$ and using that $F_{-n}\equiv F_n\pmod 2$ for all $n$, we get
\begin{eqnarray*}
H_{0,k} & = & 
	\left|\begin{matrix} F_{k+1} & F_k\\ 
				   F_{-k+1} & F_{-k+2}
	\end{matrix}\right| 
	\equiv 
	\left|\begin{matrix} F_{k+1} & F_k\\ 
				    F_{k-1} & F_{k-2}
	\end{matrix}\right|\pmod 2 \\[10pt]
	&\equiv&  F_{k+1}^2-F_k F_{k-1}\pmod 2\\
 	& \equiv &  (F_k+F_{k-1})^2-F_kF_{k-1}\pmod 2 \\
	&\equiv &    F_k^2+F_{k-1}^2+F_kF_{k-1}\pmod 2\equiv 1,
\end{eqnarray*}

\begin{eqnarray*}
H_{1,k}  &=&  \left|\begin{matrix} 
				F_{k+1} & F_k\\ 
				F_{-k+2} & F_{-k+3}
		 \end{matrix}\right| \\
		 &\equiv& 
		 \left|\begin{matrix} 
		 		F_{k+1} & F_k\\ 
				F_{k-2} & F_k
		\end{matrix}\right| \\
		&\equiv& F_k(F_{k+1}-F_{k-2})\equiv 0\pmod 2,
\end{eqnarray*}
and 
\begin{eqnarray*}
H_{2,k}  &=&  \left|\begin{matrix} 
				F_{k+1} & F_k\\ 
				F_{3-k} & F_{4-k}
			\end{matrix}\right| \\
		&\equiv& 
		\left|\begin{matrix} 
				F_{k+1} & F_k\\ 
				F_k & F_{k-1}
		\end{matrix}\right|\pmod 2\\
		&\equiv& F_{k+1}F_{k-1}-F_k^2\equiv 1\pmod 2. 
\end{eqnarray*} 
This completes the proof.  
\end{proof}

We recall the Legendre symbol 
$$\legendre{5}{q}=
\begin{cases} 
      1 & \text{ if } q \equiv \pm 1 \bmod 5 \\
      -1 & \text{ if } q \equiv \pm 2 \bmod 5 \\
       0 & \text{ if } q \equiv 0 \bmod 5 \\
\end{cases}
$$
as well as the following well-known fact.
\begin{lemma} \label{Mod5FiboLucas}
Let $p$ be a prime. Then 
$$F_{\left(p- \legendre{5}{p}\right)}\equiv 0 \bmod p.$$
\end{lemma}
Since $F_a\mid F_b$ whenever $a\mid b$, it follows that also
$$
F_{n\left(p- \legendre{5}{p}\right)}\equiv 0 \bmod p
$$
holds for every positive integer $n$. Recall that the Fibonacci sequence is periodic modulo every positive integer $m$ with a period  sometimes denoted $\rho(m)$ and called the Pisano period.  It is also known that 
if $p$ is prime then $\rho(p)\mid p-1$ if $p\equiv \pm 1 \pmod 5$ and $\rho(p)\mid 2(p+1)$ if $p\equiv \pm 2 \pmod 5$. The following corollary is immediate. 

\begin{corollary} \label{CoroMod5FiboLucas}
Let $p$ be a prime and let $t\ge 1$ be an integer. Then
\begin{enumerate}
    \item \label{CongPart2} If  $p \equiv \pm 2 \bmod 5$, then $F_{p-(2t+1)}\equiv F_{2(t+1)} \bmod p$.
    \item \label{CongPart1} If $p \equiv \pm 1 \bmod 5$, then $F_{p-2t}    \equiv F_{2t-1}   \bmod p$.
\end{enumerate}
\end{corollary}

For example, for Part \eqref{CongPart1} we write 
$$F_{p-2t}=F_{p-1-(2t-1)}\equiv F_{-(2t-1)}\pmod p\equiv F_{2t-1}\pmod p \quad \text{ since }\quad \rho(p)\mid p-1.$$  

The following proposition shows that the central column of the triangle ${\mathcal H}$ is almost free of primes and the other two central columns are free of primes. 

\begin{proposition} \label{MidColomns} If  $t$ is a positive integer, then $H_{2t,t}=F_{t +1}^2$ and $H_{2t-1,t}=F_{t -1}F_{t+2}$.

\end{proposition} 

\begin{proof} From equation \eqref{definitionHrk1} we obtain $H_{2t,t}=F_{t -1}F_{t +2}+F_{t}^2=\big(F_{t+1}^2-F_{t}^2\big)+F_{t}^2$.
Also from equation \eqref{definitionHrk1} we obtain $H_{2t-1,t}=F_{t -1}F_{t+1}+F_{t}F_{t-1}=F_{t -1}(F_{t+1}+F_{t})=F_{t-1}F_{t+2}$.
\end{proof}

\begin{theorem} \label{ThmPrimeModePrime} Let $p$ be a prime number and let $t>0$ be an integer, then the following hold: 
\begin{enumerate}
  \item \label{PropoCongPart1} If $p \equiv \pm 1 \bmod 5$, then  $H_{p,2t}   \equiv 2 H_{4(t-1),2(t-1)} \equiv 2 F_{2 t-1}^2  \bmod p$.
  \item \label{PropoCongPart2} If $p \equiv \pm 2 \bmod 5$, then  $H_{p,2t+1} \equiv 2 H_{4t,2t}-1       \equiv 2 F_{2 t+1}^2-1\bmod p$.
\end{enumerate}
\end{theorem}

\begin{proof} For part \eqref{PropoCongPart1}, by equation \eqref{definitionHrk1} we know that $H_{p,2t}= F_{2 t-1} F_{p-2(t-1)}+F_{2 t} F_{p-2t}$.
Using Corollary \ref{CoroMod5FiboLucas} we get that 
$$
H_{p,2t}\equiv F_{2t-1} F_{2t-3}+F_{2t} F_{2t-1} \equiv F_{2t-1}(F_{2t-3}+F_{2t})\pmod p\equiv 2F_{2t-1}^2\pmod p.
$$
As for Part \eqref{PropoCongPart2}, from the definition of $H_{r,k}$ we have $H_{p,2t+1}= F_{2t}F_{p-(2t-1)}+F_{2t+1}F_{p-(2t+1)}$, 

Let us take 
$F_{p-(2t-1)}=F_{(p+1)-2t}=(\alpha^{p+1-2t}-\beta^{p+1-2t})/\sqrt{5}$, where $\alpha$, $\beta$ are the golden section and its conjugate.
Since $p=\pm 2 \bmod 5$, we have $\alpha^p\equiv \beta \bmod p$. So, $\alpha^{p+1}\equiv \alpha \beta \equiv -1 \bmod p$ and the same goes for $\beta^{p+1}\equiv -1 \pmod p$. Thus, 
$$F_{p-(2t-1)}=F_{p+1-2t}\equiv -1 (\alpha^{-2t}-\beta^{-2t})/\sqrt{5}\equiv -F_{-2t}\equiv F_{2t} \bmod p.$$
This together with Corollary \ref{CoroMod5FiboLucas} implies that 
$$
H_{p,2t+1}\equiv F_{2t}F_{2t}+F_{2t+1}F_{2(t+1)} \bmod p.
$$ 
It remains to check that 
$$
F_{2t}^2+F_{2t+1}F_{2t+2}=2F_{2t+1}^2-1
$$
which holds since it is equivalent to $F_{2t+1}(2F_{2t+1}-F_{2t+2})-F_{2t}^2=1$, which in turn is equivalent to $F_{2t+1}F_{2t-1}-F_{2t}^2=1$, a well-known identity. 
\end{proof}

As a corollary we have that for a  prime number $p \equiv \pm 1 \bmod 5$ and $p> 2  F_{2n+1}^2$, then  for $n=1,2,3$, and $4$, it holds that    
$H_{p,2n} \equiv 2  F_{2n-1}^2  \bmod p$. As a second corollary we have that for a prime number $p \equiv \pm 2 \bmod 5$ and $p> 2  F_{2n-1}^2-1$,   
then for $n=1,2,3,4$, and $5$, it holds  that $H_{p,2n-1} \equiv 2  F_{2n-1}^2-1  \bmod p$.

\begin{lemma}\label{LemmaStarPorperty} For $1\le k\le n$ and positive integer $t$ we have: 
\begin{enumerate}
  \item \label{LemmaStarPorpertyPart1} $H_{n+ t,k+ t}=F_{1+ t} H_{n,k}+F_{t} H_{n-1,k-1}$. 
   
 \item \label{LemmaStarPorpertyPart3} $H_{n+ t,k}=F_{1+ t} H_{n,k}+F_{t} H_{n-1,k}$.
 
\end{enumerate}
\end{lemma}

\begin{proof}
It is known that if $\{A_n\}_{n\ge 0}$ is a sequence obeying the Fibonacci recurrence $A_{n+2}=A_{n+1}+A_n$ for all $n\ge 0$,  then  
$A_{n+t}=F_{t+1}A_n+F_tA_{n-1}$ holds for all $t\ge 0$ and $n\ge 1$. Parts \eqref{LemmaStarPorpertyPart1} and \eqref{LemmaStarPorpertyPart3}   
follow immediately by noticing that  for fixed $n$  and $k$, the sequence $\{H_{n+t,k+t}\}_{t\ge 0}$ and the sequence $\{H_{n+t,k}\}_{t\ge 0}$ obey  
Fibonacci recurrences (for the first one, in   \eqref{eq:maineq} the second row of the determinant remains fixed as $t$ varies and for the  second one the   
first row of the determinant in  \eqref{eq:maineq} remains fixed when $t$ varies). 
\end{proof}

\begin{proposition}[Star property]\label{StarPorperty1}  Let $1\le k\le n$ be integers numbers and $p$ be an odd prime. Let  $d_p=p-\left(\frac{5}{p}\right)$.
If $H_{n,k} \equiv 0 \bmod p$, then $H_{n',k'}\equiv 0\pmod p$ for all pairs of positive integers $(n',k')$ with $n'\equiv n\pmod {d_p}$ and $k'\equiv k\pmod {d_p}$. 
\end{proposition}

\begin{proof}
Items \eqref{LemmaStarPorpertyPart1} and \eqref{LemmaStarPorpertyPart3} of Lemma \ref{LemmaStarPorperty} for $t=d_p$ together with the fact that $F_{d_p}\equiv 0\pmod p$ imply that 
$$
H_{n+d_p,k+d_p} \equiv 0 \pmod p\quad {\text{\rm and}}\quad H_{n+d_p,k}\equiv 0\pmod p.
$$ 
By induction, $H_{n+u d_p, k+u d_p}\equiv 0\pmod p$ and $H_{n+u d_p, k}\equiv 0\pmod p$ hold for all integers $u$. Letting $k'=k+u d_p$, we get that $H_{n',k'}\equiv H_{n'-u d_p,k'-ud_p}\pmod p\equiv H_{n'-ud_p,k}\pmod p$.
Since $n\equiv n'\pmod {d_{p}}$, we get that $n'-ud_p=n+v d_p$ for some integer $v$. Hence, 
$$H_{n',k'}\equiv H_{n+v d_p,k}\equiv H_{n,k}\pmod p.$$
Completing the proof.  
\end{proof}

\begin{corollary} \label{PrimesMod12} Let $p$ be a prime number. 
\begin{enumerate}
\item \label{PrimesMod12Part4}  If $p \equiv 1  \bmod 12$ or if $p \equiv 7  \bmod 12$, then $H_{p, k}\equiv 0 \bmod 2$ for  $1\le k \le p$.
\item  \label{PrimesMod12Part1} If $p \equiv 5  \bmod 12$, then $H_{p, 4t+1} \equiv 0 \bmod 3$, for $t=0,1,2, \dots, \lfloor p/4\rfloor$.
\item  \label{PrimesMod12Part2} If $p \equiv 11 \bmod 12$, then $H_{p, 4t+2} \equiv 0 \bmod 3$, for $t=0,1,2, \dots, \lfloor p/4\rfloor$.

\end{enumerate} 
\end{corollary}

\begin{proof} 
Since $p \equiv 1  \bmod 12$ or $p \equiv 7  \bmod 12$, in both cases $p \equiv 1  \bmod 3$. So, part 1 follows from Proposition \ref{CongruencesMod2}. 
Since $H_{p,1}=F_{p-1}\equiv 0 \bmod 3$ when $p \equiv 5  \bmod 12$ and $d_3=4$, part 2  follows from the Star property. 
Since $H_{p,2}=L_{p-1}\equiv 0 \bmod 3$ when $p \equiv 11  \bmod 12$ and again $d_3=4$,  part 3 follows from the Star property. 
\end{proof}

Proposition \ref{GeneralModePrime} proves that there are infinitely many composite numbers in ${\mathcal H}$. More about that later. 

\begin{proposition}\label{GeneralModePrime} Let $r, k$ be positive integers and $p$ be an odd  prime number and $d_p=p-\left(\frac{5}{p}\right)$ 
as before. 
Then $H_{r,k} \equiv 0 \bmod p$ if one of these four conditions holds:
\begin{itemize}
\item[(i)] $(r,k) \equiv (1,1)  \pmod {d_p}$; 
\item[(ii)] $(r,k) \equiv (-2,0)  \pmod {d_p}$;
\item[(iii)] $(r,k) \equiv (-2,-1) \pmod {d_p}$;   
\item[(iv)] $(r,k) \equiv (-5,-2) \pmod {d_p}$.
\end{itemize}
\end{proposition}

\begin{proof}
This is immediate since $H_{1,1}=H_{-2,0}=H_{-2,-1}=H_{-5,-2}$. 
\end{proof}

\section{Composites in ${\mathcal H}$}

In this section we prove that there are arbitrarily  large neighborhoods of ${\mathcal H}$ where all entries are composite.  
 
\begin{theorem}
For every $m\ge 1$, there are infinitely many  pairs $(n,k)$ such that $H_{n+i,k+j}$ is composite for all $i,j\in \{1,\ldots,m\}$. 
\end{theorem}

\begin{proof}
Let $p_k$ be the $k$th prime. Fix $m$ and let $3,5,\ldots,p_{m^2+1}$ be the first $m^2$ odd primes and put them in an $m\times m$ array. Thus, we label such primes in some way as $P_{i,j}$ for $1\le i,j\le m$. Use the Chinese Remainder Lemma 
to construct infinitely many pairs of positive integers $(n,k)$ such that 
$$
n\equiv -(i+2)\pmod {P_{i,j}},\quad k\equiv -j\pmod {P_{i,j}}\quad {\text{\rm for~all}}\quad 1\le i,j\le m.
$$
Then 
$$
H_{n+i,k+j}\equiv 0\pmod {F_{\gcd(n+i+2,k+j)}}\equiv 0\pmod {F_{P_{i,j}}}.
$$
Since $F_{P_{i,j}}>1$ because $P_{i,j}$ is an odd prime and $F_{P_{i,j}}\mid H_{n+i,k+j}$ for all $1\le i,j\le m$ it follows that for large $n$ and $k$, all numbers $H_{n+i,k+j}$ for $1\le i,j\le m$, are composite. The theorem is proved.  
\end{proof}

\section{Discussions, questions and tables} 

\subsection{Tables from the Star property} The Table \ref{TablaCongurences} was constructed using the Star property described in Proposition \ref{StarPorperty1}.   
For example, the first multiple  of $3$  in Table \ref{TablaCongurences} is $H_{3,2}=3$. So, from the Star property we know that there will be a multiple of $3$ located a  
distance $4$ from   $H_{3,2}$ in all directions. Since this holds for every entry of the table that satisfies this condition, inductively, we summarize 
 it in this form: every  point $H_{r,k} \in \mathcal{H}$ satisfying that $r \equiv 3 \bmod 4$ and $k \equiv 2 \bmod 4$ (for short, we write  
 $(r,k) \equiv (3,2) \bmod 4$),  is divisible by $3$. The second multiple of $3$ in Table \ref{tabla2} is $H_{5,1}=3$ (and the symmetric point $H_{5,5}$).   
 In this case,  the condition is that every point   $H_{r,k} \in \mathcal{H}$ that satisfies $(r,k) \equiv (1,1) \bmod 4$ is a multiple of $3$.  
 Note that $H_{9,5}=39$  satisfies the condition previously  discussed. By the Star property we have that every entry $H_{r,k} \in \mathcal{H}$ at distance $4$  
 (in all directions) from $H_{9,5}$  is a multiple of $3$. Thus,  
 $H_{5,1}=3$, $H_{5,5}=3$, $H_{9,1}=21$, $H_{9,9}=21$, $H_{13,5}=270$ and $H_{13,9}=270$. Similarly, we find that 
 every point $H_{r,k} \in \mathcal{H}$ that satisfies $(r,k) \equiv (2,0) \bmod 4$ is a multiple of $3$ and that every point $H_{r,k} \in \mathcal{H}$  
 that satisfies $(r,k) \equiv (2,3) \bmod 4$ is a multiple of $3$. Formally, the point $H_{r,k}\equiv  0 \bmod 3$ if one of these hold: 
 $$(r,k) \equiv (3,2) \bmod 4, \;\;  (r,k) \equiv (1,1) \bmod 4; \;\;   (r,k) \equiv (2,3) \bmod 4; \; \text{ or } \;  (r,k) \equiv (2,0) \bmod 4.$$
 This is summarized in Table \ref{TablaCongurences} line 1. 
 
 Again we can prove the results in the following discussion using mathematical induction. But for simplicity of the discussion, we leave the formality  
 for the curious reader.  From Table \ref{tabla2}, we can see that the first multiples of $5$ are these points: 
 $$H_{1,1} \text{ (trivial), } \; H_{5,3}, \;H_{6,1}, \; H_{6,5} \text{ (symmetric to $H_{6,1}$) }, \; H_{8,4}, \; \text{ and } \; H_{8,5},  \text{ where } 1\le k \le n \le 8.$$
 We can consider these points as the basic inductive step. Using the Star property,  the former discussion generalizes to this---see Proposition \ref{StarPorperty1}.
 The point $H_{r,k}\equiv  0 \bmod 5$ if one of these hold: 
 $$(r,k) \equiv (0,3) \bmod 5, \; (r,k) \equiv (1,1) \bmod 5, \;  (r,k) \equiv (3,0) \bmod 5, \; \text{ or } \; (r,k) \equiv (3,4) \bmod 5.$$ 
This is summarized in Table \ref{TablaCongurences} line 2.

Similarly, using the Star property described in Proposition \ref{StarPorperty1}, we construct a few more lines of Table \ref{TablaCongurences}. Using Proposition  \ref{StarPorperty1},
in conjunction with  Propositions  \ref{LinesWithCoposite}, \ref{CongruencesMod2}, \ref{MidColomns}, and \ref{GeneralModePrime},  
and Corollary \ref{PrimesMod12} helps us find points in the triangle that are composite.  A small sieve for the prime numbers in the triangle 
is given in Figure \ref{Dis_primes20lines}. However, this is not enough to determine whether the number of primes in the triangle is finite or infinite.  

\begin{table} [!h] \small
	\centering
	\addtolength{\tabcolsep}{-3pt} \scalebox{1}{
\begin{tabular}{|c|c||l|l|l|l|l|}
  \hline 
$p$ &   $H_{r,k}  \equiv 0 \bmod p$ & $(r,k) \bmod p\text{ or } p\pm 1$ & $(r,k) \bmod p\text{ or } p\pm 1$ & $(r,k) \bmod p\text{ or } p\pm 1$ \\ \hline \hline
$3$  & $H_{r,k}\equiv 0 \bmod 3$  & $(r,k) \equiv (3,2) \bmod 4$   & $(r,k) \equiv (1,1) \bmod 4$  &  $(r,k) \equiv (2,3) \bmod 4$  \\
        &                                              & $(h,k) \equiv  (2,0) \bmod 4$  &                                              &                                                \\   \hline
$5$  & $H_{r,k}\equiv 0 \bmod 5$  & $(r,k) \equiv (0,3) \bmod 5$   & $(r,k) \equiv (1,1) \bmod 5$  &  $(r,k) \equiv (3,4) \bmod 5$  \\
        &					      & $(r,k) \equiv (3,0) \bmod 5$   && \\  \hline
$7$  & $H_{r,k}\equiv 0 \bmod 7$  & $(r,k) \equiv (5,2) \bmod 8$   & $(r,k) \equiv (5,4) \bmod 8$  &  $(r,k) \equiv (7,3) \bmod 8$  \\  
	&					      & $(r,k) \equiv (7,5) \bmod 8$  &  $(r,k) \equiv (1,1) \bmod 8$  & $(r,k) \equiv (3,6) \bmod 8$  \\
	&  					      &$(r,k) \equiv (6,7) \bmod 8$   &  $(r,k) \equiv (6,0) \bmod 8$ &\\  \hline
$11$& $H_{r,k}\equiv 0 \bmod 11$& $(r,k) \equiv (6,2) \bmod 10$ & $(r,k) \equiv (6,5) \bmod 10$&  $(r,k) \equiv (0,4) \bmod 10$\\
        &					     &  $(r,k) \equiv (0,7) \bmod 10$ & $(r,k) \equiv (1,1) \bmod 10$ & $(r,k) \equiv (5,8) \bmod 10$\\
        	&						&  $(r,k) \equiv (8,9) \bmod 10$&  $(r,k) \equiv (8,0) \bmod 10$ &\\  \hline
$13$&$H_{r,k}\equiv 0 \bmod 13$ & $(r,k) \equiv (8,1) \bmod 14$ & $(r,k) \equiv (8,8) \bmod 14$&  $(r,k) \equiv (9,5) \bmod 14$\\  
	&						&$(r,k) \equiv (12,6) \bmod 14$ & $(r,k) \equiv (12,7) \bmod 14$& $(r,k) \equiv (1,1) \bmod 14$\\
	&  						&$(r,k) \equiv (1,8) \bmod 14$&  $(r,k) \equiv (2,5) \bmod 14$  & $(r,k) \equiv (2,12) \bmod 14$\\ 
	&						&$(r,k) \equiv (5,6) \bmod 14$&  $(r,k) \equiv (5,7) \bmod 14$&  $(r,k) \equiv (5,13) \bmod 14$ \\ 
        &                                                &$(r,k) \equiv (5,0) \bmod 14 $&$(r,k) \equiv (9,12) \bmod 14$ & $(r,k) \equiv (12,13) \bmod 14$\\
        &						&  $(r,k) \equiv (12, 0) \bmod 14$ &&\\  \hline           
$17$&$H_{r,k}  \equiv 0 \bmod 17$ & $(r,k) \equiv  (1, 1) \bmod 18 $ &$(r,k) \equiv  (10,1 ) \bmod 18 $ &$(r,k) \equiv  (10, 10) \bmod 18 $  \\  
	& 					        &$(r,k) \equiv  (13, 7) \bmod 18 $ &$(r,k) \equiv  (16,8 ) \bmod 18 $ &$(r,k) \equiv  (16,9) \bmod 18 $ \\
	& 						&$(r,k) \equiv  (1,10) \bmod 18 $ &$(r,k) \equiv  (4,7) \bmod 18 $ &$(r,k) \equiv  (4,16 ) \bmod 18 $ \\
	& 						&$(r,k) \equiv  (7,8 ) \bmod 18 $ &$(r,k) \equiv  (7, 9) \bmod 18 $ &$(r,k) \equiv  (7, 17) \bmod 18 $ \\
	& 						&$(r,k) \equiv  (7,0 ) \bmod 18 $ &$(r,k) \equiv  (13,16 ) \bmod 18 $ & $(r,k) \equiv  (16,0) \bmod 18 $\\
	& 						&$(r,k) \equiv  (16,17) \bmod 18 $ & & \\ \hline   
$19$	&$H_{r,k}  \equiv 0 \bmod 19$ &$(r,k) \equiv  (1, 1) \bmod 18 $  &$(r,k) \equiv  (10, 2) \bmod 18$ &$(r,k) \equiv  (10, 9) \bmod 18$  \\  
	& 						&$(r,k) \equiv  (14, 5) \bmod 18$ &$(r,k) \equiv  (14, 10) \bmod 18$ &$(r,k) \equiv  (15, 3) \bmod 18$ \\
	& 						& $(r,k) \equiv  (15, 13) \bmod 18$&$(r,k) \equiv  (16, 17) \bmod 18$ &$(r,k) \equiv  (16, 0) \bmod 18 $ \\
	& 						&$(r,k) \equiv  (0, 7) \bmod 18 $ &$(r,k) \equiv  (0, 12) \bmod 18 $ & \\
	& 						&$(r,k) \equiv  (4, 8) \bmod 18 $ &$(r,k) \equiv  (4, 15) \bmod 18 $ &$(r,k) \equiv  (13, 16) \bmod 18$ \\
	& 						& $(r,k) \equiv  (17, 4) \bmod 18$&$(r,k) \equiv  (17, 14) \bmod 18$ & \\ \hline   
\end{tabular}}
\smallskip
\caption{Distribution of Primes in normal Hosoya triangle.} \label{TablaCongurences}
\end{table}

This table gives rise to the Figure \ref{Dis_primes20lines}. It shows the distribution of the primes in the partial triangle $\Delta_{20}$.

\begin{figure}[h!]
  \includegraphics[width=10cm]{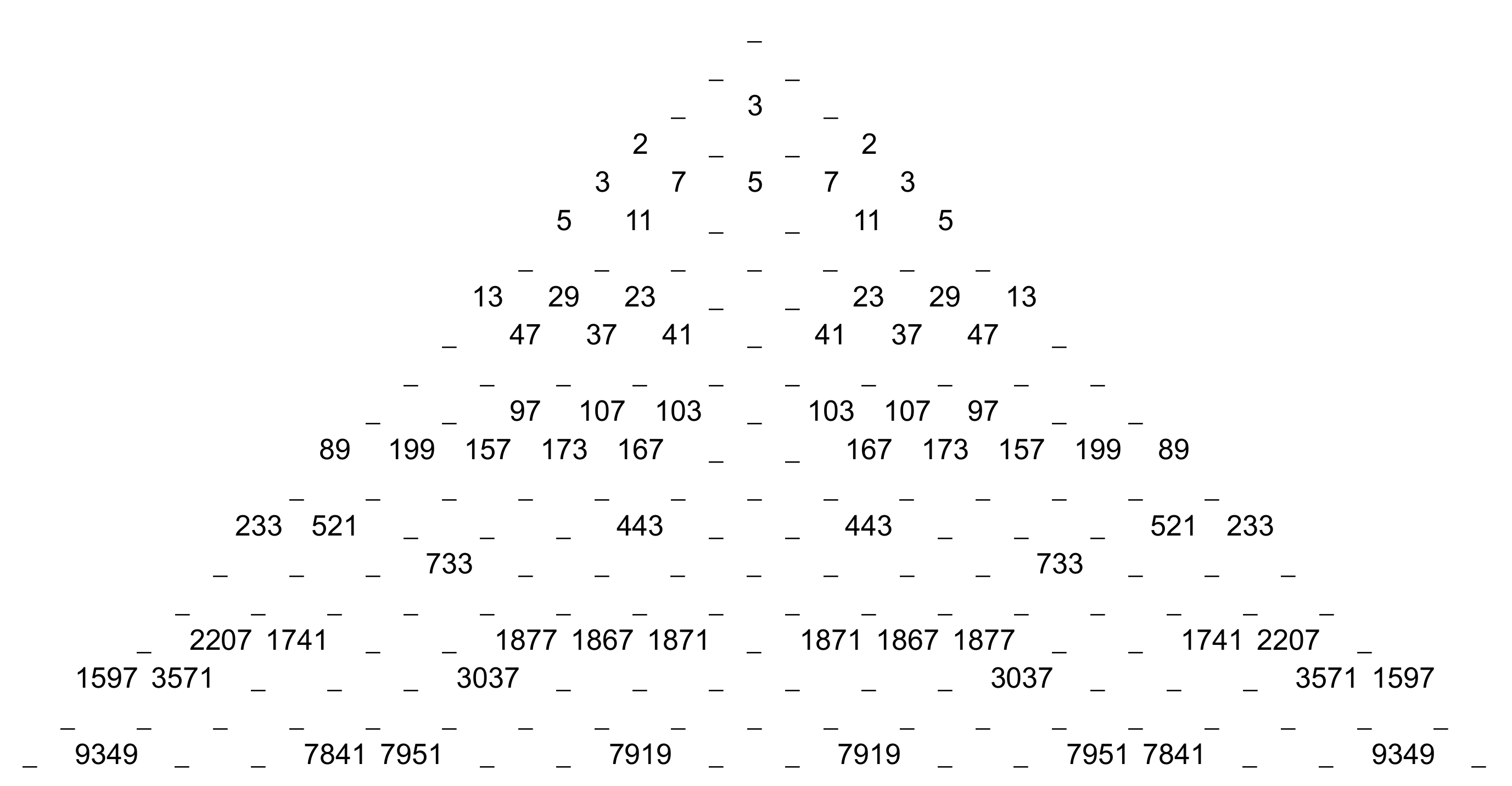}
  \caption{Primes in the first 20 lines in the triangle}\label{Dis_primes20lines}
\end{figure}

\subsection{Lines free of primes}

We now focus our discussion on lines of columns free of primes.  For example, Proposition \ref{MidColomns} tells us that the three central 
columns of the triangle are formed by non-prime numbers, except $3$ and $5$. This shows that there are no lines free of non-prime numbers.  
Proposition \ref{CongruencesMod2} shows that lines in position $r=3t+1$ are free of primes, for $t>1$.   Table \ref{TablaCongurences}  tells us that 
the points in position $k=4s+1$ of every line  in position $r=4t+1$ are divisible by 3.  Figure \ref{Dis_primes20lines} depicts the triangle with 
some lines free of primes. So, we have this question:  are there infinitely many lines  free of primes, other than the lines described in 
Proposition \ref{LinesWithCoposite}? For example, line 50 is the first of this type of lines. Thus, $H_{50,k}$ with $1\le k \le 25$ is free of primes. So, the points in line $50$ are 

\begin{table} [!h] \small
	\centering
\addtolength{\tabcolsep}{-3pt} \scalebox{1}{
\begin{tabular}{llll}  
$H_{50, 1} =  7778742049 $, &\quad $H_{50, 2} =  17393796001 $, &\quad $H_{50, 3} =  13721172195 $, &\quad $H_{50, 4} =  15123989661 $, \\
$H_{50, 5} =  14588161069 $, &\quad $H_{50, 6} =  14792829379 $, &\quad $H_{50, 7} =  14714653041 $, &\quad $H_{50, 8} =  14744513745 $, \\
$H_{50, 9} =  14733107971 $, &\quad $H_{50, 10} =  14737464589 $, &\quad $H_{50, 11} =  14735800509 $, &\quad $H_{50, 12} =  14736436131 $\\
$H_{50, 13} =  14736193345 $, &\quad $H_{50, 14} =  14736286081 $, &\quad $H_{50, 15} =  14736250659 $, &\quad $H_{50, 16} =  14736264189 $,\\
 $H_{50, 17} =  14736259021 $, &\quad $H_{50, 18} =  14736260995 $,& \quad$H_{50, 19} =  14736260241 $, &\quad $H_{50, 20} =  14736260529 $, \\
 $H_{50, 21} =  14736260419 $, &\quad $H_{50, 22} =  14736260461 $, &\quad $H_{50, 23} =  14736260445 $, &\quad $H_{50, 24} =  14736260451 $, \\
$H_{50, 25} =  14736260449 $.
\end{tabular}}
\smallskip
\end{table}

From Proposition \ref{LinesWithCoposite} we know that $H_{50, i}$ is composite, for  $i = 3$, $4$, $7$, $8$, $11$, $12$, $13$, $15$, $16$, $19$, 
$20$, $23$, $24$, $25$. Alternatively, using Table \ref{TablaCongurences} line 1 we can see that $H_{50, 3}$, $H_{50, 4}$, $H_{50, 7}$, $H_{50, 8}$, 
$H_{50, 11}$, $H_{50, 12}$, $H_{50, 15}$, $H_{50, 16}$, $H_{50, 19}$, $H_{50, 20}$, $H_{50, 23}$, and $H_{50, 24}$ are divisible by $3$. 
From line 2 we can see that $H_{50, 13}$, and $H_{50, 18}$ are divisible by $5$.  From line 4 we can see that  $H_{50, 14}$, and $H_{50, 17}$ are  
divisible by $11$. From line 5 we can see that $H_{50, 1}$  and $H_{50, 22}$ are divisible by $13$. From line 7 we can see that $H_{50, 5}$ 
and $H_{50, 10}$ are divisible by $19$.  From Proposition \ref{MidColomns} we know that $F_{13}=233$ divides $H_{50, 25}$.  Since  
$H_{8,2}=L_{7}=29$ we have that $29$ divides $H_{22,2}=H_{22,20}=L_{21}$ (since $L_{7} \mid L_{21}$). So, using the Star property we have that   
$29$ divides both $H_{50,2}$ and $H_{50,21}$. Extending Table \ref{tabla2} to more lines  we obtain that  $109 | H[50, 6]$, and  $89 | H[50, 9]$.  

We now give a few other lines that are free of primes. Lines $71$,\; $75$,\; $78$,\; $86$,\; $110$,\; $119$,\; $153$,\; $159$, \; $207$,\; $213$,\; $245$,\; $260$,\; $263$,\; $282$,\; $300$,\; $326$,\; $329$,\; $341$,\; $351$,\; $362$,\; $374$,\; $423$,\; $438$,\; $495$,\; $519$,\; $521$,\; $522$,\; $530$, 
$539$,\; $554$,\; $558$,\; $587$,\; $591$,\; $596$,\; $605$,\; $710$,\; $716$,\; $735$,\; $749$,\; $758$,\; $768$,\; $774$,\; $786$,\; $791$,\; $806$,\; $807$,\; $843$,\; $849$,  $866$,\; $869$,\; $900$,\; $903$,\; $911$,\; $918$,\; $926$,\; $930$,\; $950$,\; $960$,\; $965$,\; $966$,\; $975$,\; $986$.

If $(p,q)$ is a pair of twin primes with $p<q$, then  by Proposition \ref{CongruencesMod2} $H_{q,k}$  is free of primes for every $k$ 
(since $q$ is of the form $3t+1$).  We observe that the lines in position $p$, a prime number, have low density of prime numbers. 
For example, when $p$ is a non-twin prime we have that $p+2$ is a composite number.  So, there is a high probability that these facts hold:  
$$\gcd(p+2,k)>2\quad \text{ and } \quad \gcd(p+2,k+1)>2 \quad \text{ for } \quad1\le k \le p.$$ 
These two facts,  Propositions \ref{LinesWithCoposite} and \ref{MidColomns}, and 
Corollary \ref{PrimesMod12} imply that the number of composite numbers in line $p$ may be high. We know that $F_{p-1}$ and $L_{p-1}$ are  
composite numbers, when $p$ is a prime number. Therefore, the number of divisors $F_{p-1}$ and $L_{p-1}$ increase ($p-1$ is composite) the number of composite points in the line 
containing the point $H_{p,k}$.  This is due to the Star property and that both $F_{p-1}$ and $L_{p-1}$ are in the same line of $H_{p,k}$.

Note that using Mathematica we can verify that for $1< r\le 13461$, the lines described in Table \ref{tabla3HrkComposite} are free of primes. 
In Table \ref{tabla3HrkComposite} we did not include the case given by Propositions  \ref{LinesWithCoposite}  and \ref{CongruencesMod2} 
and we did not  include prime numbers; we only provide  composite numbers. 

\begin{table} [!h] \small
	\centering
\addtolength{\tabcolsep}{-3pt} \scalebox{1}{
\begin{tabular}{|l|r|r|r|r|r|r|r|r|r|r|r|r|r|}  \hline
Line $r$ & 329    & 351    & 519     & 539   & 591     & 605    & 749   & 807  & 965      & 975     & 1247  & 1655 &1695 \\ \hline
Line $r$ & 1821  & 2135  & 2219   & 2279  & 2375  & 2391  & 2685  & 3065 & 3269   & 3465  & 3657  & 3759 &3831 \\ \hline
Line $r$ & 4089  & 4151  & 4215   & 4269  & 4601  & 4641  & 4887  & 4941 & 5111   & 5145  & 5331  & 5415 &6005 \\ \hline
Line $r$ & 6071  & 6099  & 7077   & 7367  & 7619  & 8007  & 8309  & 8361 & 8745   & 8751  & 9411  & 9809 &9971 \\ \hline
Line $r$ & 10167 &  11157 & 11175 & 11285 & 11469 & 11481& 11487  & 11585 & 11591 & 11615 &11805& 11837 &  12035 \\ \hline
Line $r$ & 12071  & 12375 & 12471 & 12575 & 12797 &12981 & 13005  & 13047 & 13325 & 13461&&& \\ \hline
\end{tabular}}
\smallskip
\caption{Lines $r$ of $\mathcal{H}$ free of primes, where $r$ is a composite number.} \label{tabla3HrkComposite}
\end{table}

\subsection{Are there infinitely many primes of the form $H_{r,k}$ }
We believe that in the determinant Hosoya triangle there are infinitely many prime numbers of the form 
$H_{r,k}=F_{k-1}F_{r-k+2}+F_{k}F_{r-k}$. If $k=1$, then it 
gives the classic conjecture for  Fibonacci numbers. If $k=2$, then it gives the classic conjecture for Lucas numbers. 

In this section we construct a small sieve that shows that there is a large amount of primes in the triangle.

The following are some of the few primes of the form $H_{r,k}$ with $r\le 50$. \; 
$3$,\  $2$,\  $7$,\  $5$,\  $11$,\  $13$,\  $29$,\  $23$,\  $47$,\  $37$,\  $41$,\  $97$,\  $107$,\  $103$,\  $89$,\  $199$,\  $157$,\  $173$,\  $167$,\  $233$,\  $521$,\  $443$,\  $733$,\  $2207$,\   $1741$,\  $1877$,\  $1867$,\  $1871$,\  $1597$,\  $3571$,\  $3037$,\  $9349$,\  $7841$,\  $7951$,\  $7919$,\  $11933$,\  $12823$,\  $33503$,\  $28657$,\  $ 50549$,\  $55717$,\  $54497$,\  $54319$,\  $54277$,\  $54293$,\  $54287$,\  $142099$,\  $214129$,\  $236021$,\  $229963$,\  $560597$,\  $601187$,\  $514229$,\   $974249$,\  $3010349$,\  $2617513$,\  $2546669$,\  $2549863$,\  $2550329$,\  $2550407$,\  $4128959$,\  $4126697$,\  $10695127$,\  $10803367$,\  $16276621$,\  $17477021$,\  $17482189$,\  $17480681$,\  $17480791$,\  $54018521$,\  $45940907$,\  $45765017$,\  $45765227$,  $75998029$,  
$74091163$,\;  $74050573$,\; $74049641$,\;  $74049683$,\;  $180510493$,\;  $198965423$,\;  $193866917$, 
$193864477$,\  $193864603$,\  $370248451$, \ $314883661$,\  $313678093$,\  $433494437$,\ $820019509$,\  
$821683589$,\;\;  $821047967$,\;\;   $821290753$,\;\;   $821223569$, \;\;  $1328726303$,\;\;   $1328767757$,\;\;   $3478800673$,\;\;   $3478759199$,  
$2971215073$,\;  $6643838879$,\;  $5620497329$,\;  $5628750833.$

With the congruences given in Section \ref{congurencesHrk} we can determine some points in the triangle that are composite. For example, 
we analyze the points in the line 8 in Table 2 for $1\le k \le 4$. The distinct points are $H_{8,1}=13$, $H_{8,2}=29$, $H_{8,3}=23$, and $H_{8,4}=25$.  
From Proposition \ref{CongruencesMod2} we know that none of these entries are even. From the Star property or Table 
\ref{TablaCongurences} we know that  none of these points are divisible by 3, and finally from Proposition \ref{MidColomns} and 
Table \ref{TablaCongurences} line 2  we have that the only points that are divisible by five are (is)  $H_{8,4}=H_{8,5}$. Therefore, we can conclude 
that the remaining points in line 8 are prime. 

The pair $(r,k)$ in Table \ref{tablaHrkComposite} gives the coordinates of the point $H_{r,k}$ that is prime. For example, when we write $(5,3)$, 
in the above-mentioned table, we mean that the point $H_{5,3}$ is a prime number (in this case $H_{5,3}=5$). The reader can evaluate  $H_{r,k}$, 
using any of the formulas 
given in equations  \eqref{eq:maineq}, \eqref{definitionHrk1}, or \eqref{definitionHrk2}.

In Table \ref{tablaHrkComposite} we do not include the known numbers for Fibonacci and Lucas sequences.
So, the table shows only the prime numbers located after entry $2$ in the triangle.  In the appendix  there are 43 tables 
(similar to Table \ref{tablaHrkComposite}) with coordinates of primes  of the form $H_{r,k}$ for $r=1, \dots, 6000$. 
  
 The  Table \ref{tablaHrkComposite} gives a list of primes of the form $H(r,k)$, for $r=1, \dots, 200$ and $0\le k \le \lceil r/2 \rceil$. 

\begin{table} [!h] \small
	\centering
\addtolength{\tabcolsep}{-3pt} \scalebox{1}{
\begin{tabular}{|c|c|c|c|c|c|c|c|c|c|c|c|c|}  \hline
 (r, k)    & (r, k)    & (r, k)    & (r, k)    & (r, k)    & (r, k)    & (r, k)    & (r, k)    & (r, k)    &(r, k)     & (r, k)    \\ \hline \hline
 (5, 3)    & (8, 3)    & (9, 3)    & (9, 4)    & (11, 3)   & (11, 4)   & (11, 5)   & (12, 3)   & (12, 4)   & (12, 5)   & (14, 6)   \\ \hline
 (15, 4)   & (17, 3)   & (17, 6)   & (17,  7)  & (17, 8)   & (18, 6)   & (20, 5)   & (20, 6)   & (20, 9)   & (21, 3)   & (21, 8)   \\ \hline
 (23, 7)   & (24, 3)   & (24, 4)   & (24, 6)   & (24, 8)   & (24, 9)   & (24, 10)  & (24, 11)  & (26, 9)   & (27, 3)   & (27, 4)   \\ \hline
 (27, 11)  & (29, 3)   & (29, 7)   & (30, 10)  & (32, 4)   & (32, 7)   & (32, 9)   & (32, 11)  & (32, 15)  & (33, 8)   & (33, 12)  \\ \hline
 (35, 5)   & (35, 11)  & (36, 3)   & (36, 9)   & (36, 10)  & (36, 13)  & (36, 14)  & (38, 6)   & (38, 13)  & (38, 18)  & (39, 4)   \\ \hline
 (39, 8)   & (39, 12)  & (39, 15)  & (39, 17)  & (41, 3)   & (41, 4)   & (41, 12)  & (41, 15)  & (41, 19)  & (42, 6)   & (42, 13)  \\ \hline
 (44, 7)   & (44, 8)   & (44, 9)   & (44, 10)  & (44, 17)  & (45, 11)  & (45, 19)  & (47, 12)  & (47, 23)  & (48, 7)   & (48, 18)  \\ \hline
 (51, 4)   & (51, 9)   & (51, 15)  & (51, 17)  & (51, 20)  & (53, 7)   & (53, 8)   & (54, 17)  & (54, 22)  & (56, 8)   & (56, 10)  \\ \hline
 (56, 14)  & (57, 3)   & (57, 6)   & (57, 16)  & (57, 19)  & (59, 3)   & (59, 5)   & (59, 8)   & (59, 12)  & (59, 15)  & (59, 17)  \\ \hline
 (59, 20)  & (59, 25)  & (60, 21)  & (60, 22)  & (60, 25)  & (62, 6)   & (62, 9)   & (62, 17)  & (62, 26)  & (63, 16)  & (65, 10)  \\ \hline
 (65, 30)  & (66, 9)   & (66, 29)  & (66, 30)  & (68, 8)   & (68, 17)  & (69, 11)  & (72, 13)  & (72, 15)  & (72, 18)  & (74, 6)   \\ \hline
 (74, 10)  & (74, 14)  & (74, 17)  & (74, 33)  & (77, 15)  & (80, 19)  & (80, 22)  & (80, 36)  & (81, 3)   & (81, 12)  & (83, 27)  \\ \hline
 (83, 41)  & (84, 3)   & (84, 4)   & (84, 7)   & (84, 9)   & (84, 16)  & (84, 19)  & (84, 24)  & (84, 25)  & (84, 27)  & (84, 35)  \\ \hline
 (84, 40)  & (87, 8)   & (87, 28)  & (89, 16)  & (89, 30)  & (89, 43)  & (90, 9)   & (90, 10)  & (92, 19)  & (92, 31)  & (93, 3)   \\ \hline
 (95, 9)   & (95, 15)  & (95, 17)  & (95, 36)  & (95, 39)  & (95, 47)  & (96, 18)  & (96, 44)  & (98, 17)  & (98, 21)  & (98,  38) \\ \hline
 (99, 3)   & (101, 3)  & (101, 43) & (101, 47) & (102, 6)  & (102, 9)  & (102, 17) & (102, 22) & (102, 50) & (104, 38) & (104, 42) \\ \hline
 (104, 47) & (105, 12) & (105, 32) & (105, 52) & (107, 11) & (107, 23) & (107, 28) & (107, 41) & (108, 23) & (108, 52) & (111, 25) \\ \hline
 (111, 28) & (111, 40) & (113, 51) & (116, 13) & (116, 14) & (116, 29) & (116, 30) & (117, 23) & (120, 20) & (120, 42) & (122, 21) \\ \hline
 (122, 37) & (122, 49) & (122, 50) & (122, 54) & (123, 11) & (123, 27) & (123, 52) & (125, 19) & (125, 35) & (125, 46) & (126, 33) \\ \hline
 (126, 50) & (128, 16) & (128, 46) & (129, 11) & (129, 12) & (129, 39) & (131, 25) & (132, 18) & (132, 22) & (132, 23) & (132, 25) \\ \hline
 (132, 32) & (132, 35) & (134, 54) & (134, 61) & (135, 4)  & (135, 64) & (137, 20) & (137, 26) & (137, 42) & (138, 57) & (140, 25) \\ \hline
 (140, 66) & (141, 23) & (141, 40) & (141, 70) & (143, 7)  & (144, 9)  & (144, 41) & (144, 59) & (146, 13) & (146, 14) & (146, 49) \\ \hline
 (146, 69) & (147, 64) & (149, 22) & (149, 27) & (149, 55) & (149, 70) & (150, 22) & (150, 46) & (150, 49) & (152, 37) & (152, 45) \\ \hline
 (152, 52) & (155, 11) & (155, 37) & (156, 4)  & (156, 13) & (156, 23) & (156, 38) & (156, 60) & (156, 74) & (158, 13) & (158, 21) \\ \hline
 (158, 26) & (158, 62) & (158, 73) & (161, 10) & (161, 63) & (162, 17) & (162, 53) & (164, 16) & (164, 36) & (164, 46) & (164, 60) \\ \hline
 (164, 66) & (164, 69) & (164, 72) & (164, 81) & (165, 6)  & (165, 14) & (165, 47) & (165, 54) & (165, 71) & (167,  7) & (168, 38) \\ \hline
 (168, 48) & (168, 57) & (168, 72) & (170, 70) & (171, 57) & (171, 60) & (171, 80) & (171, 85) & (173, 31) & (173, 67) & (174, 26) \\ \hline
 (176, 19) & (176, 37) & (176, 53) & (177, 64) & (177, 79) & (179, 20) & (179, 69) & (180, 15) & (180, 16) & (180, 59) & (182, 6)  \\ \hline
 (182, 13) & (182, 57) & (182, 78) & (183, 16) & (183, 47) & (183, 76) & (185, 46) & (186, 78) & (188, 48) & (189, 7)  & (189, 40) \\ \hline
 (189, 59) & (191, 92) & (192, 9)  & (192, 22) & (192, 23) & (192, 35) & (192, 40) & (192, 46) & (192, 56) & (194, 25) & (194, 30) \\ \hline
 (194, 66) & (195, 32) & (195, 51) & (195, 52) & (195, 56) & (197, 7)  & (198, 57) & (200, 11) & (200, 22) & (200, 49) & (200, 55) \\ \hline
 (200, 85) & (200, 96) &		   &		   &		   &		   &		   &		   &		   &		   &		   \\ \hline
\end{tabular}}
\smallskip
\caption{Coordinates of prime numbers of the form $H_{r,k}$  all $3 \le k \le \lceil r/2 \rceil$.} \label{tablaHrkComposite}
\end{table}

\subsection{Density of primes in first 3000 lines of $\mathcal{H}$}  
In this section we study the distribution of prime numbers in the normal Hosoya triangle.
Through extensive experimentation we have found that the triangle has a high density of distinct prime numbers. A finite upper determinant 
Hosoya triangle with exactly $n$ lines is called a \emph{partial triangle} and it is denoted by $\Delta_r$. In Figure \ref{Dis_primes} we use  
dashed lines to show the distribution of the primes in $\Delta_{1000}$.

\begin{figure}[h!]
  \includegraphics[width=10cm]{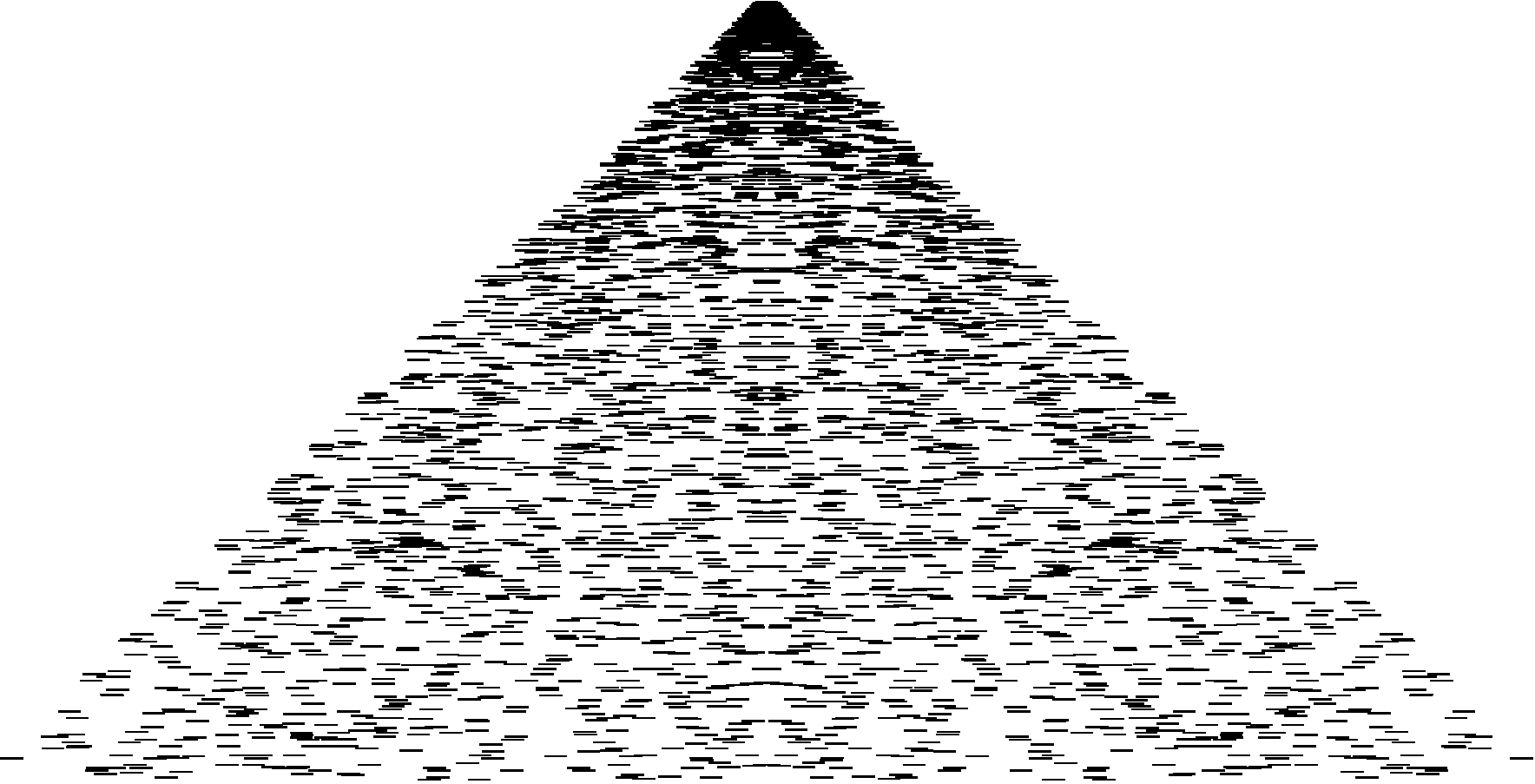}
  \caption{Primes in the first 1000 lines in the triangle}\label{Dis_primes}
\end{figure}

Let $N$ be the total number of distinct entries in the partial Hosoya triangle $\Delta_r$ with $r$ lines. It is easy to see that 
$$N=\left\lfloor \frac{1}{8} (r+1) (r-1)+\frac{1}{8} (r+3) (r+1)\right\rfloor -3=\left\lfloor \left(\frac{r+1}{2}\right)^2\right\rfloor -3.$$ 
Let $\pi_{\Delta_{r}}(N)$ be the number of distinct primes in $r$ lines (or among $N$ entries) of the triangle. In addition, let $\pi(N)$  
represent the number of primes in the set of natural numbers that are less than $N$. The following table shows the number of  
distinct primes in  the normal Hosoya triangle  $\Delta_r$ with $r$ lines.

\begin{table} [!h] \small
	\centering
	\addtolength{\tabcolsep}{-3pt} \scalebox{1}{
\begin{tabular}{|l||l|l|l|l|l|l|l|l|l|}
  \hline
   $n$                                        & 10\; & 20\;  & 30\;   & 50\;   &100\;   & 500\;     & 1000\;     & 2000\;      &3000\;\\ \hline \hline
   $N$                                       & 27\; & 107\;& 237\;& 647\;  &2547\; & 62747\; &250497\;  & 1000997\; & 2251497\; \\  \hline
   $\large{\pi}_{\Delta_{n}}(N)$ & 11   & 35   & 54     & 100     & 194    & 929        & 1876       & 3876        & 5844\\ \hline
   $\pi(N)$                                 & 9    & 28   & 51     & 118     & 372     &  6297    & 22076     & 78572       & 166169\\   \hline
\end{tabular}}
\smallskip
\caption{Distribution of Primes in normal Hosoya triangle.} \label{tabla3}
\end{table}

\section{Open questions}

The aim of this paper was to give a motivation to work a series of open questions about the primes and composite numbers within the Hosoya determinant triangle. Here we summarize three among others questions.  

\begin{enumerate} 
\item Are there infinitely many primes of either forms  $F_{k}+L_{k+1}$ or $F_{k}+L_{k-1}$?

\item Are there infinitely many primes of the form  $H_{r,k}=F_{k-1}F_{r-k+2}+F_{k}F_{r-k}$?

\item Let $R_{n}$ be the $n$-th row of the triangle. Are there infinitely many rows $R_{n}$ free of primes, with $n \not \equiv 1 \bmod 3$?   

\end{enumerate}

\section{Acknowledgement}
	
The second and fourth authors were partially supported by The Citadel Foundation.

\medskip

\noindent MSC2020: 11B39; Secondary 11B83, 11A41.

\section{Appendix; Data}

The pair $(r,k)$ in the following tables gives the coordinates of the points $H_{r,k}$ that are prime numbers. For example, the first entry of the first table  
$(5,3)$ this means that the point $H_{5,3}$ is a prime number (in this case $H_{5,3}=5$). The reader can evaluate  $H_{r,k}$, using any of the formulas  
given in equations  \eqref{eq:maineq}, \eqref{definitionHrk1}, or \eqref{definitionHrk2}.

\begin{table} [!h] \small
\centering
\addtolength{\tabcolsep}{-3pt} \scalebox{1}{
\begin{tabular}{|c|c|c|c|c|c|c|c|}  \hline
 (r, k)   	 &   (r, k)    & (r, k) 	  & (r, k)     & (r, k)   & (r, k)     & (r, k)     & (r, k)   \\ \hline \hline
 (5,3) & (8,3) & (9,3) & (9,4) & (11,3) & (11,4) & (11,5) & (12,3) \\ \hline
 (12,4) & (12,5) & (14,6) & (15,4) & (17,3) & (17,6) & (17,7) & (17,8) \\ \hline
 (18,6) & (20,5) & (20,6) & (20,9) & (21,3) & (21,8) & (23,7) & (24,3) \\ \hline
 (24,4) & (24,6) & (24,8) & (24,9) & (24,10) & (24,11) & (26,9) & (27,3) \\ \hline
 (27,4) & (27,11) & (29,3) & (29,7) & (30,10) & (32,4) & (32,7) & (32,9) \\ \hline
 (32,11) & (32,15) & (33,8) & (33,12) & (35,5) & (35,11) & (36,3) & (36,9) \\ \hline
 (36,10) & (36,13) & (36,14) & (38,6) & (38,13) & (38,18) & (39,4) & (39,8) \\
\hline
\end{tabular}}
\smallskip
\caption{Coordinates for  prime numbers of the form $H_{r,k}$  for all $3 \le k \le \lceil r/2 \rceil$.} \label{tablaHrkComposite1}
\end{table}

\begin{table} [!h] \small
	\centering
\addtolength{\tabcolsep}{-3pt} \scalebox{1}{
\begin{tabular}{|c|c|c|c|c|c|c|c|}   \hline
 (r, k)   	 &   (r, k)    & (r, k) 	  & (r, k)     & (r, k)   & (r, k)     & (r, k)     & (r, k)   \\ \hline \hline
 (191,92) & (192,9) & (192,22) & (192,23) & (192,35) & (192,40) & (192,46) & (192,56) \\ \hline
 (194,25) & (194,30) & (194,66) & (195,32) & (195,51) & (195,52) & (195,56) & (197,7) \\ \hline
 (198,57) & (200,11) & (200,22) & (200,49) & (200,55) & (200,85) & (200,96) & (201,10) \\ \hline
 (201,59) & (201,74) & (203,31) & (203,101) & (204,13) & (204,24) & (206,14) & (209,15) \\ \hline
 (209,31) & (209,68) & (209,72) & (210,6) & (210,21) & (210,85) & (212,7) & (212,9) \\ \hline
 (212,35) & (212,55) & (212,66) & (212,71) & (212,102) & (215,60) & (216,107) & (218,61) \\ \hline
  (219,15) & (219,45) & (219,105) & (221,39) & (221,103) & (222,30) & (224,22) & (224,32) \\
\hline
\end{tabular}}
\smallskip
\caption{Coordinates for  prime numbers of the form $H_{r,k}$  for all $3 \le k \le \lceil r/2 \rceil$.} \label{tablaHrkComposite2}
\end{table}

\begin{table} [!h] \small
	\centering
\addtolength{\tabcolsep}{-3pt} \scalebox{1}{
\begin{tabular}{|c|c|c|c|c|c|c|c|}   \hline
 (r, k)   	 &   (r, k)    & (r, k) 	  & (r, k)     & (r, k)   & (r, k)     & (r, k)     & (r, k)   \\ \hline \hline
(354,38) & (354,85) & (354,105) & (354,106) & (354,109) & (356,98) & (356,134) & (356,147) \\ \hline
(357,30) & (357,38) & (357,88) & (359,9) & (359,17) & (359,44) & (359,84) & (360,50) \\ \hline
(360,179) & (363,11) & (363,76) & (363,127) & (363,181) & (365,147) & (366,9) & (368,41) \\ \hline
(368,43) & (368,67) & (368,103) & (368,146) & (369,67) & (371,88) & (372,40) & (375,132) \\ \hline
(375,184) & (377,52) & (377,54) & (377,59) & (377,78) & (377,115) & (378,46) & (380,16) \\ \hline
(380,115) & (381,80) & (381,128) & (381,159) & (383,31) & (383,36) & (384,27) & (384,34) \\ \hline
(384,67) & (384,88) & (384,96) & (384,99) & (384,114) & (384,121) & (384,157) & (384,181) \\
\hline	
\end{tabular}}
\smallskip
\caption{Coordinates for  prime numbers of the form $H_{r,k}$  for all$3 \le k \le \lceil r/2 \rceil$.} \label{tablaHrkComposite3}
\end{table}

\begin{table} [!h] \small
	\centering
\addtolength{\tabcolsep}{-3pt} \scalebox{1}{
\begin{tabular}{|c|c|c|c|c|c|c|c|}   \hline
 (r, k)   	 &   (r, k)    & (r, k) 	  & (r, k)     & (r, k)   & (r, k)     & (r, k)     & (r, k)   \\ \hline \hline
(699,280) & (699,333) & (701,35) & (701,107) & (701,123) & (701,150) & (702,29) & (702,69) \\ \hline (702,93) & (704,8) & (704,62) & (704,142) & (704,213) & (704,235) & (704,237) & (704,242) \\ \hline (704,252) & (704,266) & (705,80) & (705,106) & (705,107) & (705,176) & (705,274) & (707,60) \\ \hline (707,99) & (707,117) & (707,164) & (707,207) & (708,98) & (708,317) & (708,342) & (711,257) \\ \hline (713,28) & (713,127) & (713,256) & (713,302) & (714,102) & (714,145) & (714,214) & (714,274) \\ \hline (714,282) & (717,51) & (717,326) & (719,31) & (719,116) & (719,135) & (720,42) & (720,166) \\ \hline (720,182) & (720,221) & (722,325) & (723,27) & (723,97) & (725,179) & (725,224) & (725,255) \\
\hline
\end{tabular}}
\smallskip
\caption{Coordinates for  prime numbers of the form $H_{r,k}$  for all$3 \le k \le \lceil r/2 \rceil$.} \label{tablaHrkComposite4}
\end{table}

\begin{table} [!h] \small
	\centering
\addtolength{\tabcolsep}{-3pt} \scalebox{1}{
\begin{tabular}{|c|c|c|c|c|c|c|c|}   \hline
 (r, k)   	 &   (r, k)    & (r, k) 	  & (r, k)     & (r, k)   & (r, k)     & (r, k)     & (r, k)   \\ \hline \hline
(873,143) & (873,311) & (875,387) & (875,437) & (876,168) & (876,193) & (876,228) & (876,250) \\ \hline (876,429) & (878,421) & (879,228) & (879,412) & (881,263) & (881,339) & (881,355) & (882,57) \\ \hline (882,113) & (882,141) & (882,301) & (882,410) & (884,27) & (884,66) & (884,100) & (884,233) \\ \hline (884,237) & (884,302) & (885,95) & (885,344) & (887,401) & (888,118) & (888,383) & (890,26) \\ \hline (891,27) & (891,79) & (891,120) & (891,149) & (891,237) & (891,308) & (893,203) & (893,223) \\ \hline (893,271) & (893,391) & (894,270) & (896,94) & (896,258) & (896,430) & (897,212) & (897,336) \\ \hline (897,368) & (899,367) & (899,405) & (899,417) & (902,94) & (902,161) & (902,217) & (902,393) \\
\hline
\end{tabular}}
\smallskip
\caption{Coordinates for  prime numbers of the form $H_{r,k}$  for all$3 \le k \le \lceil r/2 \rceil$.} \label{tablaHrkComposite4}
\end{table}

\begin{table} [!h] \small
	\centering
\addtolength{\tabcolsep}{-3pt} \scalebox{1}{
\begin{tabular}{|c|c|c|c|c|c|c|c|}   \hline
 (r, k)   	 &   (r, k)    & (r, k) 	  & (r, k)     & (r, k)   & (r, k)     & (r, k)     & (r, k)   \\ \hline \hline
(1001,38) & (1001,99) & (1001,394) & (1001,447) & (1001,467) & (1002,201) & (1004,227) & (1004,400) \\ \hline (1005,107) & (1005,191) & (1005,192) & (1005,231) & (1007,127) & (1007,209) & (1007,231) & (1007,329) \\ \hline (1008,123) & (1008,473) & (1010,289) & (1011,505) & (1014,21) & (1014,93) & (1014,150) & (1014,305) \\ \hline (1014,337) & (1014,405) & (1016,154) & (1016,388) & (1017,15) & (1017,127) & (1017,192) & (1017,206) \\ \hline (1017,240) & (1017,343) & (1017,352) & (1019,111) & (1019,209) & (1019,356) & (1019,489) & (1020,100) \\ \hline (1020,396) & (1020,479) & (1022,397) & (1025,150) & (1025,214) & (1025,435) & (1025,447) & (1028,31) \\ \hline (1028,38) & (1028,121) & (1028,231) & (1028,436) & (1029,40) & (1031,192) & (1031,300) & (1032,36) \\
\hline
\end{tabular}}
\smallskip
\caption{Coordinates for  prime numbers of the form $H_{r,k}$  for all$3 \le k \le \lceil r/2 \rceil$.} \label{tablaHrkComposite4}
\end{table}

\begin{table} [!h] \small
	\centering
\addtolength{\tabcolsep}{-3pt} \scalebox{1}{
\begin{tabular}{|c|c|c|c|c|c|c|c|}   \hline
 (r, k)   	 &   (r, k)    & (r, k) 	  & (r, k)     & (r, k)   & (r, k)     & (r, k)     & (r, k)   \\ \hline \hline
 (1161,202) & (1161,439) & (1161,499) & (1163,192) & (1163,273) & (1163,353) & (1163,387) & (1164,38) \\ \hline (1164,70) & (1164,386) & (1164,526) & (1167,36) & (1167,436) & (1169,38) & (1169,63) & (1169,67) \\ \hline (1169,155) & (1169,316) & (1172,169) & (1172,256) & (1172,264) & (1172,279) & (1172,293) & (1172,368) \\ \hline (1172,380) & (1172,457) & (1172,541) & (1173,118) & (1175,9) & (1175,31) & (1175,52) & (1175,199) \\ \hline (1175,356) & (1175,360) & (1175,377) & (1175,416) & (1175,479) & (1175,532) & (1175,552) & (1178,197) \\ \hline (1178,241) & (1178,366) & (1179,203) & (1179,297) & (1179,436) & (1181,19) & (1181,87) & (1181,187) \\ \hline (1182,150) & (1182,165) & (1182,365) & (1182,418) & (1184,108) & (1184,223) & (1184,416) & (1184,420) \\
\hline
\end{tabular}}
\smallskip
\caption{Coordinates for  prime numbers of the form $H_{r,k}$  for all$3 \le k \le \lceil r/2 \rceil$.} \label{tablaHrkComposite4}
\end{table}

\begin{table} [!h] \small
	\centering
\addtolength{\tabcolsep}{-3pt} \scalebox{1}{
\begin{tabular}{|c|c|c|c|c|c|c|c|}   \hline
 (r, k)   	 &   (r, k)    & (r, k) 	  & (r, k)     & (r, k)   & (r, k)     & (r, k)     & (r, k)   \\ \hline \hline
 (1305,507) & (1305,636) & (1308,207) & (1308,441) & (1310,242) & (1311,119) & (1311,279) & (1311,384) \\ \hline (1314,78) & (1314,166) & (1314,338) & (1314,526) & (1314,589) & (1314,634) & (1316,90) & (1316,108) \\ \hline (1316,244) & (1316,343) & (1316,353) & (1316,468) & (1317,136) & (1317,158) & (1317,203) & (1317,480) \\ \hline (1319,457) & (1320,96) & (1320,110) & (1320,486) & (1322,109) & (1322,146) & (1322,214) & (1322,325) \\ \hline (1322,573) & (1325,515) & (1326,289) & (1328,403) & (1328,407) & (1328,458) & (1328,593) & (1329,279) \\ \hline (1329,382) & (1331,524) & (1331,553) & (1331,623) & (1332,302) & (1332,461) & (1334,86) & (1334,181) \\ \hline (1334,313) & (1334,566) & (1337,179) & (1337,236) & (1337,426) & (1337,579) & (1337,586) & (1338,158) \\
\hline
\end{tabular}}
\smallskip
\caption{Coordinates for  prime numbers of the form $H_{r,k}$  for all$3 \le k \le \lceil r/2 \rceil$.} \label{tablaHrkComposite4}
\end{table}

\begin{table} [!h] \small
	\centering
\addtolength{\tabcolsep}{-3pt} \scalebox{1}{
\begin{tabular}{|c|c|c|c|c|c|c|c|}   \hline
 (r, k)   	 &   (r, k)    & (r, k) 	  & (r, k)     & (r, k)   & (r, k)     & (r, k)     & (r, k)   \\ \hline \hline
 (1473,636) & (1475,99) & (1475,404) & (1476,103) & (1476,288) & (1476,350) & (1476,419) & (1476,517) \\ \hline (1476,557) & (1478,642) & (1479,60) & (1479,88) & (1479,161) & (1479,200) & (1479,624) & (1481,190) \\ \hline (1481,428) & (1481,460) & (1481,732) & (1482,246) & (1482,386) & (1482,457) & (1484,9) & (1484,220) \\ \hline (1484,257) & (1484,635) & (1484,669) & (1484,686) & (1485,192) & (1485,632) & (1487,81) & (1487,143) \\ \hline (1487,297) & (1487,303) & (1487,639) & (1488,3) & (1488,151) & (1488,328) & (1488,603) & (1490,561) \\ \hline (1490,605) & (1491,224) & (1491,275) & (1491,473) & (1493,288) & (1494,106) & (1494,134) & (1494,206) \\ \hline (1496,197) & (1497,26) & (1497,80) & (1497,183) & (1497,651) & (1499,640) & (1502,18) & (1502,214) \\
\hline
\end{tabular}}
\smallskip
\caption{Coordinates for  prime numbers of the form $H_{r,k}$  for all$3 \le k \le \lceil r/2 \rceil$.} \label{tablaHrkComposite4}
\end{table}

\begin{table} [!h] \small
	\centering
\addtolength{\tabcolsep}{-3pt} \scalebox{1}{
\begin{tabular}{|c|c|c|c|c|c|c|c|}   \hline
 (r, k)   	 &   (r, k)    & (r, k) 	  & (r, k)     & (r, k)   & (r, k)     & (r, k)     & (r, k)   \\ \hline \hline
 (1632,462) & (1632,489) & (1634,62) & (1634,210) & (1635,52) & (1635,796) & (1637,499) & (1637,547) \\ \hline (1637,798) & (1638,437) & (1638,473) & (1640,10) & (1640,22) & (1641,255) & (1641,503) & (1641,628) \\ \hline (1641,670) & (1643,551) & (1643,593) & (1643,772) & (1644,286) & (1644,340) & (1644,382) & (1644,461) \\ \hline (1644,664) & (1644,718) & (1644,806) & (1646,734) & (1649,123) & (1649,275) & (1649,443) & (1649,539) \\ \hline (1649,562) & (1649,583) & (1649,691) & (1650,502) & (1650,705) & (1652,4) & (1652,108) & (1652,211) \\ \hline (1652,450) & (1652,695) & (1652,727) & (1656,69) & (1656,227) & (1656,514) & (1656,640) & (1658,533) \\ \hline (1658,778) & (1658,781) & (1659,111) & (1659,215) & (1659,485) & (1661,190) & (1661,787) & (1662,462) \\
\hline
\end{tabular}}
\smallskip
\caption{Coordinates for  prime numbers of the form $H_{r,k}$  for all$3 \le k \le \lceil r/2 \rceil$.} \label{tablaHrkComposite4}
\end{table}

\begin{table} [!h] \small
	\centering
\addtolength{\tabcolsep}{-3pt} \scalebox{1}{
\begin{tabular}{|c|c|c|c|c|c|c|c|}   \hline
 (r, k)   	 &   (r, k)    & (r, k) 	  & (r, k)     & (r, k)   & (r, k)     & (r, k)     & (r, k)   \\ \hline \hline
 (1805,671) & (1806,190) & (1806,417) & (1806,538) & (1808,437) & (1808,463) & (1808,657) & (1809,200) \\ \hline (1809,854) & (1811,340) & (1811,529) & (1812,115) & (1812,387) & (1812,465) & (1812,530) & (1812,781) \\ \hline (1812,797) & (1815,240) & (1815,449) & (1817,506) & (1818,22) & (1818,382) & (1820,226) & (1820,266) \\ \hline (1820,319) & (1820,720) & (1823,231) & (1824,47) & (1824,375) & (1824,628) & (1824,676) & (1826,430) \\ \hline (1826,498) & (1826,697) & (1826,817) & (1826,858) & (1827,27) & (1827,47) & (1827,60) & (1827,539) \\ \hline (1829,43) & (1832,44) & (1833,302) & (1833,442) & (1833,627) & (1835,467) & (1835,816) & (1836,77) \\ \hline (1836,447) & (1836,639) & (1836,724) & (1836,864) & (1839,113) & (1839,241) & (1839,404) & (1839,695) \\
\hline
\end{tabular}}
\smallskip
\caption{Coordinates for  prime numbers of the form $H_{r,k}$  for all$3 \le k \le \lceil r/2 \rceil$.} \label{tablaHrkComposite4}
\end{table}

\begin{table} [!h] \small
	\centering
\addtolength{\tabcolsep}{-3pt} \scalebox{1}{
\begin{tabular}{|c|c|c|c|c|c|c|c|}   \hline
 (r, k)   	 &   (r, k)    & (r, k) 	  & (r, k)     & (r, k)   & (r, k)     & (r, k)     & (r, k)   \\ \hline \hline
 (1805,671) & (1806,190) & (1806,417) & (1806,538) & (1808,437) & (1808,463) & (1808,657) & (1809,200) \\ \hline (1809,854) & (1811,340) & (1811,529) & (1812,115) & (1812,387) & (1812,465) & (1812,530) & (1812,781) \\ \hline (1812,797) & (1815,240) & (1815,449) & (1817,506) & (1818,22) & (1818,382) & (1820,226) & (1820,266) \\ \hline (1820,319) & (1820,720) & (1823,231) & (1824,47) & (1824,375) & (1824,628) & (1824,676) & (1826,430) \\ \hline (1826,498) & (1826,697) & (1826,817) & (1826,858) & (1827,27) & (1827,47) & (1827,60) & (1827,539) \\ \hline (1829,43) & (1832,44) & (1833,302) & (1833,442) & (1833,627) & (1835,467) & (1835,816) & (1836,77) \\ \hline (1836,447) & (1836,639) & (1836,724) & (1836,864) & (1839,113) & (1839,241) & (1839,404) & (1839,695) \\
\hline
\end{tabular}}
\smallskip
\caption{Coordinates for  prime numbers of the form $H_{r,k}$  for all$3 \le k \le \lceil r/2 \rceil$.} \label{tablaHrkComposite4}
\end{table}

\begin{table} [!h] \small
	\centering
\addtolength{\tabcolsep}{-3pt} \scalebox{1}{
\begin{tabular}{|c|c|c|c|c|c|c|c|}   \hline
 (r, k)   	 &   (r, k)    & (r, k) 	  & (r, k)     & (r, k)   & (r, k)     & (r, k)     & (r, k)   \\ \hline \hline
 (2001,3) & (2001,318) & (2001,550) & (2001,620) & (2001,772) & (2003,31) & (2003,501) & (2004,91) \\ \hline (2004,214) & (2004,434) & (2004,556) & (2004,961) & (2006,37) & (2006,814) & (2006,870) & (2007,72) \\ \hline (2007,481) & (2007,544) & (2009,159) & (2009,219) & (2009,347) & (2009,502) & (2009,618) & (2009,714) \\ \hline (2009,790) & (2009,988) & (2012,50) & (2012,485) & (2012,534) & (2012,668) & (2012,739) & (2012,789) \\ \hline (2012,882) & (2012,973) & (2013,603) & (2015,231) & (2015,335) & (2015,489) & (2015,801) & (2016,474) \\ \hline (2018,378) & (2019,200) & (2019,276) & (2019,371) & (2019,420) & (2019,623) & (2019,703) & (2019,853) \\ \hline (2021,775) & (2021,823) & (2022,694) & (2025,91) & (2025,106) & (2025,282) & (2025,315) & (2025,934) \\
\hline
\end{tabular}}
\smallskip
\caption{Coordinates for  prime numbers of the form $H_{r,k}$  for all$3 \le k \le \lceil r/2 \rceil$.} \label{tablaHrkComposite4}
\end{table}

\begin{table} [!h] \small
	\centering
\addtolength{\tabcolsep}{-3pt} \scalebox{1}{
\begin{tabular}{|c|c|c|c|c|c|c|c|}   \hline
 (r, k)   	 &   (r, k)    & (r, k) 	  & (r, k)     & (r, k)   & (r, k)     & (r, k)     & (r, k)   \\ \hline \hline
 (2157,208) & (2157,368) & (2157,478) & (2157,695) & (2157,1040) & (2159,545) & (2159,696) & (2160,471) \\ \hline (2162,5) & (2162,606) & (2162,662) & (2162,886) & (2163,133) & (2163,703) & (2163,767) & (2163,1056) \\ \hline (2163,1057) & (2165,294) & (2165,367) & (2165,696) & (2165,1014) & (2166,317) & (2166,550) & (2166,693) \\ \hline (2166,793) & (2166,854) & (2168,193) & (2168,317) & (2169,339) & (2169,631) & (2171,812) & (2172,154) \\ \hline (2172,373) & (2172,967) & (2174,246) & (2174,278) & (2175,241) & (2175,505) & (2177,75) & (2177,279) \\ \hline (2177,763) & (2177,767) & (2178,713) & (2180,95) & (2180,885) & (2180,979) & (2181,430) & (2181,583) \\ \hline (2181,587) & (2181,795) & (2183,372) & (2184,133) & (2184,468) & (2184,927) & (2184,1055) & (2186,889) \\
\hline
\end{tabular}}
\smallskip
\caption{Coordinates for  prime numbers of the form $H_{r,k}$  for all$3 \le k \le \lceil r/2 \rceil$.} \label{tablaHrkComposite4}
\end{table}

\begin{table} [!h] \small
	\centering
\addtolength{\tabcolsep}{-3pt} \scalebox{1}{
\begin{tabular}{|c|c|c|c|c|c|c|c|}   \hline
 (r, k)   	 &   (r, k)    & (r, k) 	  & (r, k)     & (r, k)   & (r, k)     & (r, k)     & (r, k)   \\ \hline \hline
 (2334,405) & (2334,482) & (2334,670) & (2334,1010) & (2334,1077) & (2334,1129) & (2336,898) & (2337,647) \\ \hline (2339,117) & (2339,692) & (2340,52) & (2340,1022) & (2340,1031) & (2340,1086) & (2342,453) & (2342,482) \\ \hline (2342,622) & (2342,761) & (2342,1034) & (2342,1161) & (2343,912) & (2345,176) & (2345,220) & (2345,876) \\ \hline (2348,322) & (2348,397) & (2348,486) & (2348,637) & (2348,901) & (2348,1056) & (2348,1061) & (2349,430) \\ \hline (2349,555) & (2349,976) & (2351,537) & (2351,564) & (2352,686) & (2352,701) & (2352,916) & (2355,344) \\ \hline (2355,377) & (2355,665) & (2355,1052) & (2357,227) & (2357,334) & (2358,393) & (2360,182) & (2360,250) \\ \hline (2361,364) & (2361,428) & (2361,647) & (2363,53) & (2363,61) & (2363,237) & (2363,877) & (2363,1107) \\
\hline
\end{tabular}}
\smallskip
\caption{Coordinates for  prime numbers of the form $H_{r,k}$  for all$3 \le k \le \lceil r/2 \rceil$.} \label{tablaHrkComposite4}
\end{table}

\begin{table} [!h] \small
	\centering
\addtolength{\tabcolsep}{-3pt} \scalebox{1}{
\begin{tabular}{|c|c|c|c|c|c|c|c|}   \hline
 (r, k)   	 &   (r, k)    & (r, k) 	  & (r, k)     & (r, k)   & (r, k)     & (r, k)     & (r, k)   \\ \hline \hline
 (2492,323) & (2492,642) & (2492,791) & (2492,822) & (2493,142) & (2493,1088) & (2493,1171) & (2495,785) \\ \hline (2495,812) & (2495,831) & (2496,87) & (2496,414) & (2496,553) & (2496,727) & (2496,757) & (2498,338) \\ \hline (2498,853) & (2498,933) & (2499,12) & (2499,344) & (2499,596) & (2499,879) & (2499,996) & (2499,1167) \\ \hline (2501,115) & (2501,455) & (2502,717) & (2502,1082) & (2504,82) & (2504,151) & (2504,866) & (2504,1180) \\ \hline (2505,106) & (2505,166) & (2505,471) & (2505,1091) & (2505,1210) & (2507,452) & (2507,557) & (2507,567) \\ \hline (2508,153) & (2508,491) & (2508,928) & (2510,422) & (2513,783) & (2514,466) & (2514,726) & (2516,94) \\ \hline (2516,1050) & (2516,1168) & (2517,555) & (2517,688) & (2519,505) & (2519,567) & (2519,767) & (2519,785) \\
\hline
\end{tabular}}
\smallskip
\caption{Coordinates for  prime numbers of the form $H_{r,k}$  for all$3 \le k \le \lceil r/2 \rceil$.} \label{tablaHrkComposite4}
\end{table}

\begin{table} [!h] \small
	\centering
\addtolength{\tabcolsep}{-3pt} \scalebox{1}{
\begin{tabular}{|c|c|c|c|c|c|c|c|}   \hline
 (r, k)   	 &   (r, k)    & (r, k) 	  & (r, k)     & (r, k)   & (r, k)     & (r, k)     & (r, k)   \\ \hline \hline
 (2648,308) & (2648,321) & (2648,623) & (2648,913) & (2648,1052) & (2648,1191) & (2649,151) & (2649,851) \\ \hline (2649,1103) & (2651,233) & (2652,204) & (2652,278) & (2652,503) & (2652,567) & (2652,576) & (2652,917) \\ \hline (2652,919) & (2652,1000) & (2652,1016) & (2652,1216) & (2654,177) & (2654,1293) & (2655,711) & (2657,10) \\ \hline (2657,30) & (2657,211) & (2657,527) & (2657,971) & (2657,1231) & (2660,250) & (2660,356) & (2660,565) \\ \hline (2661,1104) & (2661,1208) & (2664,373) & (2664,505) & (2664,796) & (2664,869) & (2664,1037) & (2666,829) \\ \hline (2666,1033) & (2667,444) & (2667,805) & (2667,1225) & (2669,879) & (2672,470) & (2672,1060) & (2673,411) \\ \hline (2673,511) & (2673,672) & (2673,676) & (2673,832) & (2673,972) & (2675,97) & (2675,745) & (2675,1152) \\
\hline
\end{tabular}}
\smallskip
\caption{Coordinates for  prime numbers of the form $H_{r,k}$  for all$3 \le k \le \lceil r/2 \rceil$.} \label{tablaHrkComposite4}
\end{table}

\begin{table} [!h] \small
	\centering
\addtolength{\tabcolsep}{-3pt} \scalebox{1}{
\begin{tabular}{|c|c|c|c|c|c|c|c|}   \hline
 (r, k)   	 &   (r, k)    & (r, k) 	  & (r, k)     & (r, k)   & (r, k)     & (r, k)     & (r, k)   \\ \hline \hline
 (2810,362) & (2810,1382) & (2811,472) & (2811,473) & (2811,600) & (2811,628) & (2811,1080) & (2811,1097) \\ \hline (2811,1223) & (2814,225) & (2814,993) & (2816,450) & (2817,162) & (2817,316) & (2817,442) & (2817,702) \\ \hline (2817,862) & (2817,1294) & (2819,645) & (2820,232) & (2820,682) & (2820,720) & (2823,332) & (2823,428) \\ \hline (2823,648) & (2823,1116) & (2825,355) & (2826,554) & (2826,877) & (2828,88) & (2828,677) & (2828,827) \\ \hline (2828,983) & (2828,1032) & (2828,1268) & (2829,863) & (2829,1059) & (2831,417) & (2831,480) & (2832,222) \\ \hline (2832,747) & (2832,879) & (2832,1162) & (2834,234) & (2834,329) & (2834,622) & (2834,778) & (2835,72) \\ \hline (2837,584) & (2837,958) & (2838,1302) & (2840,31) & (2840,191) & (2840,430) & (2840,540) & (2840,866) \\
\hline
\end{tabular}}
\smallskip
\caption{Coordinates for  prime numbers of the form $H_{r,k}$  for all $3 \le k \le \lceil r/2 \rceil$.} \label{tablaHrkComposite4}
\end{table}

\begin{table} [!h] \small
	\centering
\addtolength{\tabcolsep}{-3pt} \scalebox{1}{
\begin{tabular}{|c|c|c|c|c|c|c|c|}   \hline
 (r, k)   	 &   (r, k)    & (r, k) 	  & (r, k)     & (r, k)   & (r, k)      \\ \hline \hline
(3002,302) & (3002,490) & (3002,785) & (3002,810) & (3003,352) & (3003,511) \\ \hline (3003,1108) & (3005,399) & (3005,574) & (3005,1119) & (3006,1454) & (3008,17) \\ \hline (3008,362) & (3008,401) & (3008,1348) & (3009,419) & (3009,582) & (3009,923) \\ \hline (3011,604) & (3012,371) & (3012,562) & (3012,849) & (3012,1003) & (3012,1250) \\ \hline (3012,1432) & (3015,95) & (3015,865) & (3017,587) & (3017,998) & (3017,1182) \\ \hline (3017,1447) & (3018,541) & (3020,410) & (3020,752) & (3020,1096) & (3020,1366) \\ \hline (3021,72) & (3021,310) & (3021,343) & (3021,408) & (3021,647) & (3021,1070) \\
\hline
\end{tabular}}
\smallskip
\caption{Coordinates for  prime numbers of the form $H_{r,k}$  for all $3 \le k \le \lceil r/2 \rceil$.} \label{tablaHrkComposite4}
\end{table}

\begin{table} [!h] \small
	\centering
\addtolength{\tabcolsep}{-3pt} \scalebox{1}{
\begin{tabular}{|c|c|c|c|c|c|c|c|}   \hline
 (r, k)   	 &   (r, k)    & (r, k) 	  & (r, k)     & (r, k)   & (r, k)      \\ \hline \hline
(3128,253) & (3128,677) & (3128,742) & (3128,807) & (3128,1366) & (3129,1508) \\ \hline (3131,583) & (3131,1317) & (3132,70) & (3132,81) & (3132,536) & (3132,671) \\ \hline (3132,1047) & (3132,1186) & (3132,1323) & (3132,1443) & (3134,666) & (3135,472) \\ \hline (3135,1431) & (3137,399) & (3137,522) & (3137,1015) & (3137,1512) & (3137,1543) \\ \hline (3138,162) & (3138,1458) & (3140,715) & (3143,32) & (3143,1388) & (3144,998) \\ \hline (3147,961) & (3147,1485) & (3147,1512) & (3149,1351) & (3150,826) & (3152,157) \\ \hline (3152,612) & (3152,842) & (3152,886) & (3152,1072) & (3153,328) & (3153,411) \\
\hline
\end{tabular}}
\smallskip
\caption{Coordinates for  prime numbers of the form $H_{r,k}$  for all $3 \le k \le \lceil r/2 \rceil$.} \label{tablaHrkComposite4}
\end{table}

\begin{table} [!h] \small
	\centering
\addtolength{\tabcolsep}{-3pt} \scalebox{1}{
\begin{tabular}{|c|c|c|c|c|c|c|c|}   \hline
 (r, k)   	 &   (r, k)    & (r, k) 	  & (r, k)     & (r, k)   & (r, k)      \\ \hline \hline
(3248,797) & (3249,462) & (3251,48) & (3251,452) & (3251,1035) & (3251,1324) \\ \hline (3252,155) & (3252,274) & (3252,292) & (3252,674) & (3252,948) & (3252,1056) \\ \hline (3252,1093) & (3252,1437) & (3254,669) & (3254,942) & (3255,1396) & (3257,171) \\ \hline (3257,448) & (3257,506) & (3257,1168) & (3258,218) & (3258,481) & (3260,661) \\ \hline (3260,925) & (3260,1045) & (3260,1129) & (3260,1570) & (3261,238) & (3263,196) \\ \hline (3264,41) & (3264,59) & (3264,279) & (3264,493) & (3264,744) & (3264,881) \\ \hline (3264,1086) & (3267,449) & (3267,801) & (3267,859) & (3267,1103) & (3270,1101) \\
\hline
\end{tabular}}
\smallskip
\caption{Coordinates for  prime numbers of the form $H_{r,k}$  for all $3 \le k \le \lceil r/2 \rceil$.} \label{tablaHrkComposite4}
\end{table}

\begin{table} [!h] \small
	\centering
\addtolength{\tabcolsep}{-3pt} \scalebox{1}{
\begin{tabular}{|c|c|c|c|c|c|c|c|}   \hline
 (r, k)   	 &   (r, k)    & (r, k) 	  & (r, k)     & (r, k)   & (r, k)      \\ \hline \hline
(3360,1490) & (3362,129) & (3362,373) & (3363,756) & (3363,1043) & (3363,1496) \\ \hline (3363,1576) & (3366,334) & (3366,538) & (3366,634) & (3368,237) & (3368,528) \\ \hline (3368,696) & (3368,733) & (3368,956) & (3368,1568) & (3369,107) & (3369,570) \\ \hline (3369,582) & (3369,1427) & (3371,545) & (3371,649) & (3371,992) & (3372,996) \\ \hline (3372,1019) & (3372,1195) & (3372,1254) & (3374,978) & (3374,1009) & (3375,12) \\ \hline (3375,412) & (3377,643) & (3377,707) & (3377,1411) & (3378,942) & (3378,1522) \\ \hline (3380,80) & (3380,396) & (3380,720) & (3381,240) & (3383,188) & (3383,583) \\
\hline
\end{tabular}}
\smallskip
\caption{Coordinates for  prime numbers of the form $H_{r,k}$  for all $3 \le k \le \lceil r/2 \rceil$.} \label{tablaHrkComposite4}
\end{table}

\begin{table} [!h] \small
	\centering
\addtolength{\tabcolsep}{-3pt} \scalebox{1}{
\begin{tabular}{|c|c|c|c|c|c|c|c|}   \hline
 (r, k)   	 &   (r, k)    & (r, k) 	  & (r, k)     & (r, k)   & (r, k)      \\ \hline \hline
(3488,591) & (3488,1137) & (3488,1438) & (3489,191) & (3489,732) & (3489,978) \\ \hline (3489,1116) & (3489,1668) & (3491,59) & (3491,65) & (3491,1083) & (3492,1054) \\ \hline (3492,1133) & (3492,1404) & (3492,1497) & (3492,1711) & (3494,81) & (3494,446) \\ \hline (3494,601) & (3495,1044) & (3497,427) & (3497,1083) & (3498,102) & (3498,421) \\ \hline (3498,633) & (3500,1222) & (3500,1670) & (3501,243) & (3501,280) & (3501,303) \\ \hline (3501,384) & (3501,472) & (3501,1144) & (3501,1192) & (3504,524) & (3504,726) \\ \hline (3504,1434) & (3506,1418) & (3507,64) & (3507,92) & (3507,259) & (3507,872) \\
\hline
\end{tabular}}
\smallskip
\caption{Coordinates for  prime numbers of the form $H_{r,k}$  for all $3 \le k \le \lceil r/2 \rceil$.} \label{tablaHrkComposite4}
\end{table}

\begin{table} [!h] \small
	\centering
\addtolength{\tabcolsep}{-3pt} \scalebox{1}{
\begin{tabular}{|c|c|c|c|c|c|c|c|}   \hline
 (r, k)   	 &   (r, k)    & (r, k) 	  & (r, k)     & (r, k)   & (r, k)      \\ \hline \hline
(3615,1512) & (3617,306) & (3617,519) & (3617,535) & (3617,1348) & (3618,417) \\ \hline (3618,942) & (3620,546) & (3620,731) & (3620,1056) & (3621,150) & (3621,595) \\ \hline (3621,1207) & (3621,1760) & (3623,383) & (3624,730) & (3624,1395) & (3626,1569) \\ \hline (3627,308) & (3627,503) & (3627,508) & (3627,585) & (3627,748) & (3627,1236) \\ \hline (3627,1599) & (3627,1617) & (3627,1648) & (3627,1783) & (3629,1040) & (3629,1590) \\ \hline (3630,1090) & (3632,17) & (3632,303) & (3632,1132) & (3632,1328) & (3632,1663) \\ \hline (3633,1322) & (3635,160) & (3635,1096) & (3635,1445) & (3635,1465) & (3636,828) \\
\hline
\end{tabular}}
\smallskip
\caption{Coordinates for  prime numbers of the form $H_{r,k}$  for all $3 \le k \le \lceil r/2 \rceil$.} \label{tablaHrkComposite4}
\end{table}

\begin{table} [!h] \small
	\centering
\addtolength{\tabcolsep}{-3pt} \scalebox{1}{
\begin{tabular}{|c|c|c|c|c|c|c|c|}   \hline
 (r, k)   	 &   (r, k)    & (r, k) 	  & (r, k)     & (r, k)   & (r, k)      \\ \hline \hline
(3744,886) & (3744,1048) & (3744,1207) & (3746,1210) & (3747,1105) & (3747,1129) \\ \hline (3747,1323) & (3747,1777) & (3749,523) & (3749,1615) & (3752,88) & (3752,415) \\ \hline (3752,940) & (3752,1183) & (3752,1247) & (3752,1471) & (3753,788) & (3753,1326) \\ \hline (3753,1596) & (3755,401) & (3755,789) & (3755,856) & (3755,1236) & (3755,1492) \\ \hline (3755,1541) & (3755,1824) & (3756,420) & (3756,934) & (3756,1094) & (3756,1273) \\ \hline (3756,1717) & (3758,837) & (3761,244) & (3762,258) & (3762,1358) & (3764,127) \\ \hline (3764,536) & (3764,1241) & (3764,1474) & (3764,1550) & (3764,1768) & (3764,1853) \\
\hline
\end{tabular}}
\smallskip
\caption{Coordinates for  prime numbers of the form $H_{r,k}$  for all $3 \le k \le \lceil r/2 \rceil$.} \label{tablaHrkComposite4}
\end{table}

\begin{table} [!h] \small
	\centering
\addtolength{\tabcolsep}{-3pt} \scalebox{1}{
\begin{tabular}{|c|c|c|c|c|c|c|c|}   \hline
 (r, k)   	 &   (r, k)    & (r, k) 	  & (r, k)     & (r, k)   & (r, k)      \\ \hline \hline
(3863,1577) & (3864,37) & (3864,731) & (3864,748) & (3864,995) & (3864,1502) \\ \hline (3864,1534) & (3864,1722) & (3864,1870) & (3866,713) & (3866,1050) & (3867,684) \\ \hline (3867,1089) & (3867,1699) & (3869,123) & (3869,943) & (3869,982) & (3869,1115) \\ \hline (3872,83) & (3872,1258) & (3872,1321) & (3872,1453) & (3873,636) & (3873,688) \\ \hline (3873,702) & (3873,1088) & (3873,1403) & (3873,1822) & (3875,1776) & (3876,418) \\ \hline (3876,599) & (3876,1880) & (3879,188) & (3879,480) & (3879,900) & (3879,1644) \\ \hline (3881,214) & (3881,240) & (3881,492) & (3881,828) & (3881,1392) & (3882,797) \\
\hline
\end{tabular}}
\smallskip
\caption{Coordinates for  prime numbers of the form $H_{r,k}$  for all $3 \le k \le \lceil r/2 \rceil$.} \label{tablaHrkComposite4}
\end{table}

\begin{table} [!h] \small
	\centering
\addtolength{\tabcolsep}{-3pt} \scalebox{1}{
\begin{tabular}{|c|c|c|c|c|c|c|c|}   \hline
 (r, k)   	 &   (r, k)    & (r, k) 	  & (r, k)     & (r, k)   & (r, k)      \\ \hline \hline
(3980,1355) & (3981,1558) & (3983,1208) & (3984,40) & (3984,135) & (3984,489) \\ \hline (3984,670) & (3984,1257) & (3984,1495) & (3986,489) & (3986,1033) & (3987,144) \\ \hline (3987,632) & (3987,1100) & (3987,1452) & (3989,790) & (3989,838) & (3989,1928) \\ \hline (3990,1369) & (3992,31) & (3992,1399) & (3993,383) & (3993,678) & (3993,1358) \\ \hline (3995,397) & (3995,401) & (3995,887) & (3995,1957) & (3996,83) & (3996,808) \\ \hline (3996,1068) & (3996,1118) & (3996,1583) & (3996,1970) & (3999,928) & (4001,87) \\ \hline (4001,1303) & (4001,1588) & (4001,1918) & (4002,333) & (4002,1901) & (4004,278) \\
\hline
\end{tabular}}
\smallskip
\caption{Coordinates for  prime numbers of the form $H_{r,k}$  for all $3 \le k \le \lceil r/2 \rceil$.} \label{tablaHrkComposite4}
\end{table}

\begin{table} [!h] \small
	\centering
\addtolength{\tabcolsep}{-3pt} \scalebox{1}{
\begin{tabular}{|c|c|c|c|c|c|c|c|}   \hline
 (r, k)   	 &   (r, k)    & (r, k) 	  & (r, k)     & (r, k)   & (r, k)      \\ \hline \hline
(4101,400) & (4101,447) & (4101,919) & (4101,955) & (4101,1563) & (4103,12) \\ \hline (4103,188) & (4103,1396) & (4104,28) & (4104,135) & (4104,472) & (4104,843) \\ \hline (4104,1292) & (4104,1435) & (4104,1567) & (4106,1589) & (4106,1878) & (4107,176) \\ \hline (4107,561) & (4107,1159) & (4107,1216) & (4107,1740) & (4107,1744) & (4107,1781) \\ \hline (4107,1783) & (4109,912) & (4109,1438) & (4109,1711) & (4109,1719) & (4110,81) \\ \hline (4110,865) & (4112,756) & (4112,794) & (4112,903) & (4112,1577) & (4112,1917) \\ \hline (4113,1178) & (4113,1902) & (4115,395) & (4115,2021) & (4116,799) & (4116,883) \\
\hline
\end{tabular}}
\smallskip
\caption{Coordinates for  prime numbers of the form $H_{r,k}$  for all $3 \le k \le \lceil r/2 \rceil$.} \label{tablaHrkComposite4}
\end{table}

\begin{table} [!h] \small
	\centering
\addtolength{\tabcolsep}{-3pt} \scalebox{1}{
\begin{tabular}{|c|c|c|c|c|c|c|c|}   \hline
 (r, k)   	 &   (r, k)    & (r, k) 	  & (r, k)     & (r, k)   & (r, k)      \\ \hline \hline
(4220,260) & (4220,979) & (4220,1222) & (4220,1372) & (4220,1511) & (4220,1592) \\ \hline (4221,703) & (4221,907) & (4221,990) & (4221,1103) & (4221,1894) & (4223,447) \\ \hline (4223,452) & (4224,846) & (4224,2070) & (4227,164) & (4227,1340) & (4227,1524) \\ \hline (4229,102) & (4230,1162) & (4232,117) & (4232,266) & (4232,562) & (4232,821) \\ \hline (4232,1129) & (4232,1640) & (4232,2090) & (4233,1348) & (4235,1105) & (4235,1887) \\ \hline (4235,1969) & (4236,658) & (4236,1510) & (4236,1908) & (4238,241) & (4238,662) \\ \hline (4239,887) & (4239,1588) & (4241,1198) & (4241,1652) & (4242,613) & (4242,1634) \\
\hline
\end{tabular}}
\smallskip
\caption{Coordinates for  prime numbers of the form $H_{r,k}$  for all $3 \le k \le \lceil r/2 \rceil$.} \label{tablaHrkComposite4}
\end{table}

\begin{table} [!h] \small
	\centering
\addtolength{\tabcolsep}{-3pt} \scalebox{1}{
\begin{tabular}{|c|c|c|c|c|c|c|c|}   \hline
 (r, k)   	 &   (r, k)    & (r, k) 	  & (r, k)     & (r, k)   & (r, k)      \\ \hline \hline
(4350,249) & (4350,1550) & (4352,731) & (4352,976) & (4352,1528) & (4352,1543) \\ \hline (4352,1640) & (4352,1642) & (4352,2132) & (4353,991) & (4355,1525) & (4355,1824) \\ \hline (4356,464) & (4356,877) & (4356,1864) & (4358,201) & (4359,1072) & (4359,1208) \\ \hline (4361,970) & (4362,390) & (4364,79) & (4364,347) & (4364,639) & (4364,896) \\ \hline (4364,1519) & (4364,1847) & (4364,2047) & (4365,1712) & (4365,2096) & (4367,888) \\ \hline (4367,2152) & (4368,446) & (4368,938) & (4368,1658) & (4368,1923) & (4371,843) \\ \hline (4371,1449) & (4371,1452) & (4373,992) & (4373,1691) & (4374,554) & (4374,590) \\
\hline
\end{tabular}}
\smallskip
\caption{Coordinates for  prime numbers of the form $H_{r,k}$  for all $3 \le k \le \lceil r/2 \rceil$.} \label{tablaHrkComposite4}
\end{table}

\begin{table} [!h] \small
	\centering
\addtolength{\tabcolsep}{-3pt} \scalebox{1}{
\begin{tabular}{|c|c|c|c|c|c|c|c|}   \hline
 (r, k)   	 &   (r, k)    & (r, k) 	  & (r, k)     & (r, k)   & (r, k)      \\ \hline \hline
(4487,183) & (4487,537) & (4487,617) & (4487,647) & (4487,1924) & (4488,383) \\ \hline (4488,732) & (4488,1398) & (4488,1866) & (4488,1987) & (4490,66) & (4490,1686) \\ \hline (4490,2222) & (4491,344) & (4491,392) & (4491,812) & (4491,877) & (4491,2028) \\ \hline (4491,2217) & (4493,91) & (4493,198) & (4493,1011) & (4494,1041) & (4494,1629) \\ \hline (4494,1753) & (4496,237) & (4496,2038) & (4497,12) & (4497,343) & (4497,838) \\ \hline (4497,1340) & (4497,2138) & (4497,2172) & (4499,33) & (4499,1324) & (4500,1012) \\ \hline (4500,1759) & (4502,1734) & (4503,233) & (4503,263) & (4503,1396) & (4505,970) \\
\hline
\end{tabular}}
\smallskip
\caption{Coordinates for  prime numbers of the form $H_{r,k}$  for all $3 \le k \le \lceil r/2 \rceil$.} \label{tablaHrkComposite4}
\end{table}

\begin{table} [!h] \small
	\centering
\addtolength{\tabcolsep}{-3pt} \scalebox{1}{
\begin{tabular}{|c|c|c|c|c|c|c|c|}   \hline
 (r, k)   	 &   (r, k)    & (r, k) 	  & (r, k)     & (r, k)   & (r, k)      \\ \hline \hline
(4608,463) & (4608,603) & (4608,1053) & (4608,1557) & (4610,842) & (4611,1017) \\ \hline (4611,2213) & (4613,63) & (4613,783) & (4613,2091) & (4614,1369) & (4614,1762) \\ \hline (4614,1838) & (4614,2090) & (4616,490) & (4616,1178) & (4616,1530) & (4616,1578) \\ \hline (4616,2044) & (4617,380) & (4617,2047) & (4617,2128) & (4617,2227) & (4619,867) \\ \hline (4619,929) & (4619,1475) & (4619,1769) & (4619,2207) & (4620,1119) & (4622,801) \\ \hline (4622,2253) & (4625,939) & (4625,1551) & (4625,1604) & (4626,1250) & (4626,1837) \\ \hline (4628,726) & (4628,1482) & (4628,1893) & (4628,2288) & (4629,303) & (4629,1774) \\
\hline
\end{tabular}}
\smallskip
\caption{Coordinates for  prime numbers of the form $H_{r,k}$  for all $3 \le k \le \lceil r/2 \rceil$.} \label{tablaHrkComposite4}
\end{table}

\begin{table} [!h] \small
	\centering
\addtolength{\tabcolsep}{-3pt} \scalebox{1}{
\begin{tabular}{|c|c|c|c|c|c|c|c|}   \hline
 (r, k)   	 &   (r, k)    & (r, k) 	  & (r, k)     & (r, k)   & (r, k)      \\ \hline \hline
(4719,1848) & (4721,1298) & (4722,214) & (4722,797) & (4722,1730) & (4724,259) \\ \hline (4724,722) & (4724,784) & (4724,1044) & (4724,1644) & (4724,1848) & (4724,2148) \\ \hline (4724,2166) & (4724,2269) & (4724,2317) & (4725,291) & (4725,2095) & (4727,329) \\ \hline (4727,2300) & (4728,168) & (4728,457) & (4728,1048) & (4728,1357) & (4728,1818) \\ \hline (4730,1662) & (4731,404) & (4731,1208) & (4731,1433) & (4731,1437) & (4733,2307) \\ \hline (4734,513) & (4736,508) & (4736,743) & (4736,1569) & (4736,2194) & (4736,2204) \\ \hline (4739,328) & (4739,1317) & (4739,1929) & (4739,1996) & (4739,2139) & (4740,1630) \\
\hline
\end{tabular}}
\smallskip
\caption{Coordinates for  prime numbers of the form $H_{r,k}$  for all $3 \le k \le \lceil r/2 \rceil$.} \label{tablaHrkComposite4}
\end{table}

\begin{table} [!h] \small
	\centering
\addtolength{\tabcolsep}{-3pt} \scalebox{1}{
\begin{tabular}{|c|c|c|c|c|c|c|c|}   \hline
 (r, k)   	 &   (r, k)    & (r, k) 	  & (r, k)     & (r, k)   & (r, k)      \\ \hline \hline
(4848,702) & (4848,881) & (4848,1237) & (4848,1508) & (4848,2146) & (4850,506) \\ \hline (4850,1610) & (4850,2266) & (4851,1937) & (4851,1964) & (4853,423) & (4853,651) \\ \hline (4853,2398) & (4854,1341) & (4854,1710) & (4854,2282) & (4856,1220) & (4856,2059) \\ \hline (4856,2144) & (4857,222) & (4857,806) & (4857,1006) & (4857,2074) & (4859,249) \\ \hline (4859,419) & (4859,1621) & (4859,1808) & (4859,2265) & (4859,2284) & (4859,2323) \\ \hline (4860,410) & (4860,499) & (4860,2202) & (4865,170) & (4865,319) & (4865,406) \\ \hline (4865,1012) & (4866,1129) & (4866,1974) & (4868,96) & (4868,792) & (4868,1123) \\
\hline
\end{tabular}}
\smallskip
\caption{Coordinates for  prime numbers of the form $H_{r,k}$  for all $3 \le k \le \lceil r/2 \rceil$.} \label{tablaHrkComposite4}
\end{table}

\begin{table} [!h] \small
	\centering
\addtolength{\tabcolsep}{-3pt} \scalebox{1}{
\begin{tabular}{|c|c|c|c|c|c|c|c|}   \hline
 (r, k)   	 &   (r, k)    & (r, k) 	  & (r, k)     & (r, k)   & (r, k)      \\ \hline \hline
(4965,504) & (4965,920) & (4967,1612) & (4967,1884) & (4968,1786) & (4968,2277) \\ \hline (4970,510) & (4971,395) & (4971,508) & (4971,1385) & (4971,1840) & (4973,1471) \\ \hline (4973,1751) & (4974,577) & (4974,638) & (4974,1641) & (4974,1990) & (4974,2098) \\ \hline (4974,2146) & (4976,340) & (4976,377) & (4976,914) & (4976,1859) & (4977,927) \\ \hline (4977,1019) & (4977,1736) & (4979,993) & (4980,1090) & (4982,366) & (4985,496) \\ \hline (4985,1412) & (4985,1770) & (4985,1806) & (4985,2051) & (4988,762) & (4988,1183) \\ \hline (4988,1263) & (4991,359) & (4991,1919) & (4991,2084) & (4992,20) & (4992,222) \\
\hline
\end{tabular}}
\smallskip
\caption{Coordinates for  prime numbers of the form $H_{r,k}$  for all $3 \le k \le \lceil r/2 \rceil$.} \label{tablaHrkComposite4}
\end{table}

\begin{table} [!h] \small
	\centering
\addtolength{\tabcolsep}{-3pt} \scalebox{1}{
\begin{tabular}{|c|c|c|c|c|c|c|c|}   \hline
 (r, k)   	 &   (r, k)    & (r, k) 	  & (r, k)     & (r, k)   & (r, k)      \\ \hline \hline
(5100,106) & (5100,120) & (5100,1546) & (5100,1582) & (5103,948) & (5103,1876) \\ \hline (5103,2497) & (5105,504) & (5106,870) & (5106,1533) & (5108,1797) & (5108,2496) \\ \hline (5112,1320) & (5112,1390) & (5112,1918) & (5112,2026) & (5112,2068) & (5112,2114) \\ \hline (5112,2368) & (5114,38) & (5114,1533) & (5114,2173) & (5115,751) & (5115,2145) \\ \hline (5117,2155) & (5117,2262) & (5118,1301) & (5118,2281) & (5120,581) & (5120,652) \\ \hline (5120,799) & (5120,2000) & (5121,332) & (5121,802) & (5121,804) & (5121,887) \\ \hline (5121,1192) & (5121,1619) & (5121,1982) & (5121,2043) & (5121,2380) & (5123,803) \\
\hline
\end{tabular}}
\smallskip
\caption{Coordinates for  prime numbers of the form $H_{r,k}$  for all $3 \le k \le \lceil r/2 \rceil$.} \label{tablaHrkComposite4}
\end{table}

\begin{table} [!h] \small
	\centering
\addtolength{\tabcolsep}{-3pt} \scalebox{1}{
\begin{tabular}{|c|c|c|c|c|c|c|c|}   \hline
 (r, k)   	 &   (r, k)    & (r, k) 	  & (r, k)     & (r, k)   & (r, k)      \\ \hline \hline
(5232,2443) & (5235,449) & (5237,1355) & (5238,813) & (5238,1018) & (5240,530) \\ \hline (5240,1359) & (5240,1381) & (5240,2085) & (5240,2505) & (5241,15) & (5243,312) \\ \hline (5243,1497) & (5243,2091) & (5244,12) & (5244,123) & (5244,411) & (5244,1398) \\ \hline (5244,2287) & (5244,2523) & (5246,1297) & (5246,2113) & (5247,1048) & (5247,1732) \\ \hline (5247,2177) & (5247,2207) & (5249,15) & (5249,56) & (5249,1215) & (5249,1675) \\ \hline (5249,1840) & (5249,2120) & (5249,2260) & (5249,2375) & (5250,1921) & (5250,2102) \\ \hline (5252,102) & (5252,209) & (5252,429) & (5252,1354) & (5252,1925) & (5253,1467) \\
\hline
\end{tabular}}
\smallskip
\caption{Coordinates for  prime numbers of the form $H_{r,k}$  for all $3 \le k \le \lceil r/2 \rceil$.} \label{tablaHrkComposite4}
\end{table}

\begin{table} [!h] \small
	\centering
\addtolength{\tabcolsep}{-3pt} \scalebox{1}{
\begin{tabular}{|c|c|c|c|c|c|c|c|}   \hline
 (r, k)   	 &   (r, k)    & (r, k) 	  & (r, k)     & (r, k)   & (r, k)      \\ \hline \hline
(5354,486) & (5355,645) & (5355,700) & (5355,1817) & (5355,2340) & (5357,1486) \\ \hline (5357,1879) & (5358,2517) & (5361,635) & (5361,2334) & (5363,67) & (5363,553) \\ \hline (5363,2251) & (5364,717) & (5364,1255) & (5364,2216) & (5367,137) & (5367,1696) \\ \hline (5367,1991) & (5369,48) & (5369,76) & (5369,798) & (5369,955) & (5369,1111) \\ \hline (5369,1207) & (5369,2318) & (5369,2403) & (5370,962) & (5372,155) & (5372,1008) \\ \hline (5372,1239) & (5372,1468) & (5372,2332) & (5373,223) & (5373,347) & (5373,663) \\ \hline (5373,927) & (5373,2307) & (5373,2512) & (5375,172) & (5375,775) & (5375,865) \\
\hline
\end{tabular}}
\smallskip
\caption{Coordinates for  prime numbers of the form $H_{r,k}$  for all $3 \le k \le \lceil r/2 \rceil$.} \label{tablaHrkComposite4}
\end{table}

\begin{table} [!h] \small
	\centering
\addtolength{\tabcolsep}{-3pt} \scalebox{1}{
\begin{tabular}{|c|c|c|c|c|c|c|c|}   \hline
 (r, k)   	 &   (r, k)    & (r, k) 	  & (r, k)     & (r, k)   & (r, k)      \\ \hline \hline
(5471,1655) & (5472,401) & (5472,705) & (5472,982) & (5472,1508) & (5472,1787) \\ \hline (5472,2101) & (5472,2377) & (5472,2571) & (5474,1722) & (5475,1945) & (5477,31) \\ \hline (5477,1099) & (5477,2623) & (5478,2281) & (5480,469) & (5480,986) & (5480,1600) \\ \hline (5481,699) & (5481,1392) & (5481,1927) & (5483,2083) & (5483,2407) & (5483,2711) \\ \hline (5484,750) & (5484,1464) & (5484,1869) & (5484,2479) & (5484,2507) & (5484,2686) \\ \hline (5486,2558) & (5487,1036) & (5487,1040) & (5487,2327) & (5489,327) & (5489,819) \\ \hline (5489,847) & (5489,1612) & (5492,129) & (5492,564) & (5492,835) & (5492,1044) \\
\hline
\end{tabular}}
\smallskip
\caption{Coordinates for  prime numbers of the form $H_{r,k}$  for all $3 \le k \le \lceil r/2 \rceil$.} \label{tablaHrkComposite4}
\end{table}

\begin{table} [!h] \small
	\centering
\addtolength{\tabcolsep}{-3pt} \scalebox{1}{
\begin{tabular}{|c|c|c|c|c|c|c|c|}   \hline
 (r, k)   	 &   (r, k)    & (r, k) 	  & (r, k)     & (r, k)   & (r, k)      \\ \hline \hline
(5579,2428) & (5580,560) & (5580,775) & (5580,1480) & (5580,1602) & (5580,1621) \\ \hline (5582,1509) & (5583,208) & (5583,228) & (5583,2436) & (5585,4) & (5585,810) \\ \hline (5585,1579) & (5585,2756) & (5586,398) & (5586,518) & (5588,326) & (5588,1876) \\ \hline (5589,951) & (5589,1375) & (5589,1760) & (5589,2462) & (5589,2734) & (5591,32) \\ \hline (5592,907) & (5592,1502) & (5592,2172) & (5592,2685) & (5592,2776) & (5595,1415) \\ \hline (5595,1961) & (5595,2447) & (5595,2571) & (5597,1167) & (5597,2435) & (5598,926) \\ \hline (5600,159) & (5600,779) & (5600,2065) & (5601,648) & (5601,1334) & (5601,1522) \\
\hline
\end{tabular}}
\smallskip
\caption{Coordinates for  prime numbers of the form $H_{r,k}$  for all $3 \le k \le \lceil r/2 \rceil$.} \label{tablaHrkComposite4}
\end{table}

\begin{table} [!h] \small
	\centering
\addtolength{\tabcolsep}{-3pt} \scalebox{1}{
\begin{tabular}{|c|c|c|c|c|c|}  \hline
 (r, k)   	 &   (r, k)    & (r, k) 	  & (r, k)     & (r, k)   & (r, k)       \\ \hline \hline
(5705,2736) & (5706,2814) & (5708,368) & (5708,743) & (5708,863) & (5708,1241) \\ \hline (5708,1411) & (5708,1802) & (5708,1856) & (5708,2243) & (5708,2647) & (5709,107) \\ \hline (5709,2339) & (5711,105) & (5711,2008) & (5711,2764) & (5712,10) & (5712,13) \\ \hline (5712,60) & (5712,900) & (5712,903) & (5712,1099) & (5712,1131) & (5712,1330) \\ \hline (5712,1489) & (5712,1563) & (5712,1634) & (5712,1816) & (5712,1985) & (5712,2169) \\ \hline (5712,2727) & (5712,2792) & (5715,44) & (5715,172) & (5715,424) & (5715,587) \\ \hline (5715,2452) & (5717,1856) & (5718,1522) & (5720,1336) & (5721,1814) & (5721,2174) \\
\hline
\end{tabular}}
\smallskip
\caption{Coordinates for  prime numbers of the form $H_{r,k}$  for all $3 \le k \le \lceil r/2 \rceil$.} \label{tablaHrkComposite4}
\end{table}

\begin{table} [!h] \small
	\centering
\addtolength{\tabcolsep}{-3pt} \scalebox{1}{
\begin{tabular}{|c|c|c|c|c|c|}  \hline
 (r, k)   	 &   (r, k)    & (r, k) 	  & (r, k)     & (r, k)   & (r, k)        \\ \hline \hline
(5901,1072) & (5901,2939) & (5903,1016) & (5904,230) & (5904,238) & (5904,346) \\ \hline (5904,518) & (5904,2343) & (5904,2559) & (5907,752) & (5907,1929) & (5909,912) \\ \hline (5909,1187) & (5909,1208) & (5909,1832) & (5910,225) & (5910,1345) & (5912,583) \\ \hline (5912,717) & (5912,1224) & (5912,1434) & (5912,1506) & (5912,1688) & (5912,1964) \\ \hline (5912,1984) & (5912,2840) & (5915,469) & (5915,1201) & (5915,2141) & (5915,2325) \\ \hline (5919,96) & (5919,473) & (5919,1276) & (5919,2584) & (5921,1842) & (5921,2378) \\ \hline (5922,82) & (5922,1065) & (5922,2225) & (5922,2614) & (5924,794) & (5924,994) \\
\hline
\end{tabular}}
\smallskip
\caption{Coordinates for  prime numbers of the form $H_{r,k}$  for all $3 \le k \le \lceil r/2 \rceil$.} \label{tablaHrkComposite4}

\end{table}

\end{document}